\documentclass[11pt]{amsart}
\usepackage{amsmath,amssymb,amsthm,skull,mathrsfs} 

\usepackage{tikz} \usetikzlibrary{matrix,arrows,decorations.pathmorphing}

\usepackage{hyperref}

\usepackage[hmarginratio=1:1, totalheight=22 true cm, totalwidth=16 true cm]{geometry}
\usepackage{color}

\title{Strongly \'{e}tale difference algebras and Babbitt's decomposition}

\author{Ivan Toma\v{s}i\'{c}}
\address{School of Mathematical Sciences, Queen Mary University of London, E1 4NS, UK}
\email{i.tomasic@qmul.ac.uk}

\author{Michael Wibmer}
\address{Department of Mathematics, University of Pennsylvania, Philadelphia PA 19104-6395, USA}\email{wibmer@math.upenn.edu}

\thanks{The authors would like to thank the London Mathematical Society for supporting our work through a \emph{Research in Pairs} grant. }

\newtheorem{theo}{Theorem}[section]
\newtheorem{lemma}[theo]{Lemma}
\newtheorem{prop}[theo]{Proposition}
\newtheorem{cor}[theo]{Corollary}
\newtheorem{defi}[theo]{Definition}
\newtheorem{rem}[theo]{Remark}
\newtheorem{fact}[theo]{Fact}

\newtheorem{conj}[theo]{Conjecture}

\theoremstyle{definition}
\newtheorem{ex}[theo]{Example}

\newcommand{\ida}{\mathfrak{a}}

\newcommand{\ld}{\operatorname{ld}}
\newcommand{\core}{\operatorname{Core}}

\renewcommand{\l}{\langle}
\renewcommand{\r}{\rangle}

\newcommand{\s}{\sigma}

\newcommand{\N}{\mathbb{N}}
\newcommand{\C}{\mathbb{C}}

\newcommand{\hs}{{{}^\sigma\!}}

\newcommand{\pis}{\pi_0^\sigma}
\newcommand{\ks}{$k$-$\s$}
\newcommand{\ssetale}{{strongly $\sigma$-\'{e}tale}}

\def\Gal{\mathscr{G}}

\begin{document}

\maketitle

\begin{abstract}
 We introduce a class of \emph{strongly \'etale difference algebras}, 
 whose role in the study of difference equations is analogous to the role of \'{e}tale algebras in the study of algebraic equations. 
 We deduce an improved version of Babbitt's decomposition theorem and we present applications to difference algebraic groups and the compatibility problem. 
\end{abstract}

\section*{Introduction}

\subsection*{Strongly $\s$-\'etale algebras and strong core}

\'{E}tale algebras play an important role in commutative algebra and algebraic geometry. In this paper we initiate the study of suitable notions of \emph{\'etale} in difference algebra and geometry by 
introducing  \emph{strongly $\s$-\'etale difference algebras} as difference algebras that are (algebraically) \'etale and \emph{$\s$-separable}. We develop a comprehensive understanding of the class of difference algebras they constitute, and consequently we establish their rugged permanence properties. 

We define the \emph{strong core} of a difference algebra $R$ over a difference field $k$ as
the union 
$$\pis(R|k)$$ 
of all strongly $\s$-\'{e}tale difference subalgebras of $R$, and we show that it has good functorial properties. In particular,  the formation of the strong core is compatible with the tensor product and extension of the base difference field.


Strongly $\sigma$-\'etale difference algebras are the best behaved among the variety of notions of \'etale extensions of difference rings studied in \cite{ive-tgs} (\'etale extensions of finite $\sigma$-type and ind-$\sigma$-\'etale extensions, with or without the $\sigma$-separability assumption) and \cite{ive-direct-tgs}  (directly-\'etale extensions). These more general notions will be studied 
in our forthcoming work.

\subsection*{Improved Babbitt's decomposition}

The framework of strongly $\s$-\'etale difference algebras and the strong core allows us to prove an enhanced variant of Babbitt's decomposition, one of the most important structure theorems in difference algebra. Babbitt's theorem was instrumental for extending specializations and in the study of compatibility of difference field extensions, see \cite[Section 5.4]{Levin:difference}. It has also been used in \cite{acfa2}, as well as \cite{ive-tgs} and \cite{ive-tgsacfa}.

The original Babbitt's decomposition (see \cite[Theorem 2.3] {Babbitt:FinitelyGeneratedPathologicalExtensionsOfDifferenceFields} or \cite[Theorem 5.4.13]{Levin:difference}) 
applies to finitely generated extensions of difference fields $L|K$ such that $K$ is inversive and the field extension $L|K$ is Galois.

The assumption that the base difference field $K$ is inversive is used in several steps of the known proofs of Babbitt's decomposition theorem and is somewhat of a hindrance for applying the theorem. We show here that the assumption that the base difference field is inversive can be dropped. This generalization is achieved by replacing the core of $L|K$ with the strong core $\pis(L|K)$.

By definition, for an arbitrary extension of difference fields $L|K$, the core of $L|K$ contains $\pis(L|K)$. On the other hand, we show that the core is \emph{$\sigma$-radicial} over the strong core.

An upshot is that we can shed light on the classical problem of compatibility of difference field extensions.   The classical compatibility theorem (see \cite[Theorem 5.4.22]{Levin:difference}) states that two extensions of difference fields $L|K$ and $L'|K$ are compatible if and only if their cores are compatible. We can actually show that $L|K$ and $L'|K$ are compatible if and only if $\pis(L|K)$ and $\pis(L'|K)$ are compatible.

\subsection*{Difference algebraic groups and connected components}

If $G$ is an (affine) algebraic group over a field $k$, then the union $\pi_0(k[G])$ of all \'{e}tale $k$-subalgebras of the coordinate ring $k[G]$ of $G$ is a Hopf-subalgebra of $k[G]$ that represents the quotient $G/G^o$ of $G$ by the connected component $G^o$ of $G$. Indeed, one may use $\pi_0(k[G])$ to define $G^o$ (see \cite[Section 6.7]{Waterhouse:IntroductiontoAffineGroupSchemes},  \cite[Chapter XIII, Def. 3.1]{Milne:BasicTheoryOfAffineGroupSchemes}).

Our initial motivation for studying the difference analog of \'{e}tale algebras was the desire to define the \emph{difference connected component} of an (affine) difference algebraic group $G$, i.e., $G$ is a group defined by algebraic difference equations. 
If $R$ is a difference Hopf algebra over a difference field $k$, the union of all difference subalgebras of $R$ that are \'{e}tale as $k$-algebras need not be a Hopf subalgebra of $R$, as shown in Example \ref{ex: core not Hopfalgebra}. In this article, however, we show that $\pis(R)$ is a Hopf subalgebra of $R$. This paves the way for the definition of the difference connected component of a difference algebraic group, see \cite[Section 4.2]{Wibmer:Habil}. 

Difference algebraic groups occur as Galois groups of linear differential or difference equations depending on a discrete parameter (\cite{DiVizioHardouinWibmer:DifferenceGaloisTheory},\cite{OvchinnikovWibmer:SigmaGalois}). We expect that, via these Galois theories, a better understanding of the components of difference algebraic groups will contribute to  a better understanding of the difference algebraic relations among the solutions of linear differential or difference equations.


\subsection*{Structure of the paper}

Let us describe the content of the article in more detail. In Section~\ref{section:one}, we discuss $\s$-separability, introduce \ssetale{} difference algebras and establish their basic properties. We define the strong core and show its functorial properties.
Finally, we explain the relation between $\pis(L|K)$ and the core of a difference field extension $L|K$.

In Section~\ref{section:babb} we establish the improved version of Babbitt's decomposition theorem (Theorem \ref{theo: Babbitt}). 

Section~\ref{section:appli} is devoted to applications of the developed theory. 
\begin{enumerate}
\item
As a contribution to the study of difference connected components of difference algebraic groups, we show that $\pis(R|k)$ is a Hopf subalgebra if $R$ is a difference Hopf algebra, respecting the appropriate finiteness properties.
\item Elaborating on the classical theory of difference field extensions, we present an application of our work to the compatibility problem.
\end{enumerate}

\section{Strongly $\s$-\'{e}tale difference algebras}\label{section:one}

We start by recalling the basic definitions and conventions from difference algebra. Standard references for difference algebra are \cite{Levin:difference} and \cite{Cohn:difference}. All rings are assumed to be commutative. A \emph{difference ring} (or \emph{$\s$-ring} \index{$\s$-ring} for short) is a ring $R$ together with a ring endomorphism $\s\colon R\to R$. A \emph{morphism of $\s$-rings} $R$ and $S$ is a morphism $\psi\colon R\to S$ of rings such that
$$
 \begin{tikzpicture} 
\matrix(m)[matrix of math nodes, row sep=2em, column sep=2em, text height=1.5ex, text depth=0.25ex]
 {  |(11)|{R} & |(12)|{S} \\
     |(21)|{R} & |(22)|{S}\\}; 
\path[->,font=\scriptsize,>=to, thin]
(11) edge node[above]{$\psi$}  (12) edge node[left]{$\sigma$} (21)
(21) edge node[above]{$\psi$} (22)
(12) edge node[right]{$\sigma$} (22);
\end{tikzpicture}
$$
commutes. In this situation, we also say that $S$ is an \emph{$R$-$\s$-algebra}. Note that in contrast to \cite{Levin:difference} and \cite{Cohn:difference} we do not require $\s\colon R\to R$ to be injective. A $\s$-ring $R$ with $\s\colon R\to R$ injective is called
\emph{$\s$-reduced}. If $\s\colon R\to R$ is an automorphism, $R$ is called \emph{inversive}.

If $R$ and $S$ are $\s$-rings such that $R$ is a subring of $S$ and the inclusion map is a morphism of $\s$-rings, we say that $R$ is a \emph{$\s$-subring} of $S$.

A \emph{difference field} \index{difference field} (or \emph{$\s$-field} \index{$\s$-field} for short) is a difference ring whose underlying ring is a field. If $K$ and $L$ are $\s$-fields such that $K$ is a $\s$-subring of $L$, we also say that $L$ is a \emph{$\s$-field extension} of $K$, or that $L|K$ is an \emph{extension of $\s$-fields}, or that $K$ is a \emph{$\s$-subfield} of $L$.

Let $L|K$ be an extension of $\s$-fields and $F\subset L$ a subset. The smallest $\s$-subfield of $L$ that contains $K$ and $F$ is denoted by
$$K\langle F\rangle.$$
As a field extension of $K$ the $\s$-field $K\langle F\rangle$ is generated by all elements of the form $\s^i(f)$ with $i\in\mathbb{N}$ and $f\in F$. We call $K\langle F\rangle$ the \emph{$\s$-field $\s$-generated by $F$ over $K$}.\index{$\s$-field $\s$-generated by $F$ over $K$} If $L=K\langle F\rangle$ for a finite subset $F$ of $L$ we say that $L$ is \emph{finitely $\s$-generated as a $\s$-field extension of $K$}.

\medskip

Let $k$ be a difference ring. A \emph{morphism of \ks-algebras} is a morphism of $k$-algebras that is also a morphism of $\s$-rings.
If $R$ and $S$ are \ks-algebras $R\otimes_k S$ is naturally a $\s$-ring by $\s(r\otimes s)=\s(r)\otimes\s(s)$ for $r\in R$ and $s\in S$. 

A \emph{\ks-subalgebra} of a \ks-algebra $R$ is a $k$-subalgebra of $R$ that is also $\s$-subring of $R$. For a subset $F\subset R$, the smallest \ks-subalgebra of $R$ that contains $F$ is denoted by
$$k\{F\}.$$
It is called the \emph{\ks-subalgebra of $R$ $\s$-generated by $F$} (over $k$). As a $k$-algebra, $k\{F\}$ is generated by all elements of the form $\s^i(f)$ with $i\in\mathbb{N}$ and $f\in F$. A \ks-algebra $R$ is called \emph{finitely $\s$-generated} (over $k$) if there exists a finite set $F\subset R$ such that $R=k\{F\}$.

Let $k$ be a $\s$-field and $R$ a \ks-algebra. Note that $\s\colon k\to k$ is a morphism of difference rings. We denote by $${\hs R}$$ the \ks-algebra obtained from $R$ by extension of scalars via $\s\colon k\to k$, i.e., ${\hs R}=R\otimes_k k$ with $k$-algebra structure map $k\to {\hs R}$ given by $\lambda\mapsto 1\otimes\lambda$. 
We note that the map $$\psi\colon {\hs R}=R\otimes_k k\to R,\ f\otimes\lambda \mapsto \s(f)\lambda$$ is a morphism of $k$-algebras.

If $V$ is a vector space over $k$ with $k$-basis $(v_i)_{i\in I}$, then
$V\otimes_k K$ is a vector space over $K$ with $K$-basis $(v_i\otimes 1)_{i\in I}$. It follows that if $(f_i)_{i\in I}$ is a $k$-basis of $R$, then $(f_i\otimes 1)_{i\in I}$ is a $k$-basis of ${\hs R}$. Therefore, if $R$ is finite dimensional as a vector space over $k$, the map $\psi$ is injective if and only if it is surjective. This fact will be used repeatedly in the proofs to follow.

For a $\s$-reduced difference ring $R$, we denote its \emph{inversive closure} (see \cite[Def. 2.1.6]{Levin:difference}) by $R^*$. This is an inversive difference ring containing $R$ such that for every $f\in R^*$, there exists an $n\geq 0$ with $\s^n(f)\in R$.

If $f$ is a univariate polynomial with coefficients in some field, and $\tau$ is an endomorphism on that field, we denote by
$${^{\tau}\!f}$$ the polynomial obtained from $f$ by applying $\tau$ to the coefficients.

\subsection{$\s$-separable $\s$-algebras} \label{subsec: sseparable salgebras}

\emph{From now on we assume that $k$ is a difference field.}

\medskip

The notion of a $\s$-separable $\s$-algebra does not appear in the standard textbooks \cite{Cohn:difference} and \cite{Levin:difference}, but it has proved useful in recent literature, including \cite[Section 5.1]{Hrushovski:elementarytheoryoffrobenius}, \cite[Section 1.5]{Wibmer:thesis}, \cite[Section 4]{DiVizioHardouinWibmer:DifferenceGaloisTheory} and \cite{ive-tgs}, \cite{ive-tgsacfa}.
In this subsection we collect basic properties of $\s$-separable $\s$-algebras that will be needed later.


\begin{defi} \label{defi: sseparable}
	A \ks-algebra $R$ is called \emph{$\s$-separable} \index{$\s$-separable} (over $k$) if $R\otimes _k K$ is $\s$-reduced for every $\s$-field extension $K|k$.
\end{defi}

Note that Definition \ref{defi: sseparable} generalizes the notion of a separable algebra. (See Corollary~\ref{cor: ssep iff sep} for a precise mathematical statement.)


\begin{prop} \label{prop: sseparable}
	The following conditions on a \ks-algebra $R$ are equivalent: 
	\begin{enumerate}
		\item $R$ is $\s$-separable over $k$.
		\item $R\otimes_k K$ is $\s$-reduced, where $K=k^*$ denotes the inversive closure of $k$.
		\item $R\otimes_ k K$ is $\s$-reduced for some inversive $\s$-field extension $K$ of $k$.		
		\item If $f_1,\ldots,f_n\in R$ are $k$-linearly independent, then also $\s(f_1),\ldots,\s(f_n)\in R$ are $k$-linearly independent. 
		\item The canonical map ${\hs R}=R\otimes_k k\to R,\ f\otimes\lambda\mapsto \s(f)\lambda$ is injective.
		\item $R$ is $\s$-reduced and linearly disjoint from $k^*$ over $k$ (inside $R^*$).
		\item $R\otimes_k S$ is $\s$-reduced for every $\s$-reduced \ks-algebra $S$. 
	\end{enumerate}
\end{prop}
\begin{proof} Cf. Proposition 1.5.2 in \cite{Wibmer:thesis} and Theorem 2, Chapter V, \S 15, No. 4 in \cite{Bourbaki:Algebra2}. The implications (i)$\Rightarrow$(ii) and (ii)$\Rightarrow$(iii) are trivial. Let us show that (iii) implies (iv). Assume $\lambda_1\s(f_1)+\cdots+\lambda_n\s(f_n)=0$ with $\lambda_1,\ldots,\lambda_n\in k$. Because $K$ is inversive, there exist $\mu_1,\ldots,\mu_n\in K$ with $\s(\mu_i)=\lambda_i$ for $i=1,\ldots,n$. Then
	$$\s(\sum f_i\otimes\mu_i)=\sum\s(f_i)\otimes\lambda_i=0\in R\otimes_k K.$$
	Because $R\otimes_k K$ is $\s$-reduced it follows that $\sum f_i\otimes\mu_i=0$. If $(h_i)_{i\in I}$ is a $k$-basis of $R$, then $(h_i\otimes 1)_{i\in I}$ is a $K$-basis of $R\otimes_k K$. Since the $f_i$'s are $k$-linearly independent, they can be extended to a $k$-basis of $R$ and it follows that $f_1\otimes 1,\ldots,f_n\otimes 1$ are $K$-linearly independent. Thus $\mu_i=0$ for $i=1,\ldots,n$ and therefore $\lambda_i=\s(\mu_i)=0$ for $i=1,\ldots,n$.
	
	Clearly (iv) and (v) are equivalent. 
	Next we will show that (iv) implies (vi). If $f\in R$ is non-zero, $\s(f)$ is $k$-linearly independent and thus non-zero. So $R$ is $\s$-reduced. Let $f_1,\ldots,f_n\in R\subset R^*$ be $k$-linearly independent and assume that $\lambda_1,\ldots,\lambda_n\in k^*$ are such that $\lambda_1f_1+\cdots+\lambda_nf_n=0$. Then there is some $m\geq 1$ such that $\s^m(\lambda_i)\in k$ for $i=1,\ldots,n$. So $\s^m(\lambda_1)\s^m(f_1)+\cdots+\s^m(\lambda_n)\s^m(f_n)=0$ is a linear dependence relation for $\s^m(f_1),\ldots,\s^m(f_n)$ over $k$. By (iv) it must be trivial. So $\s^m(\lambda_i)=0$ and therefore $\lambda_i=0$. 
	
	To see that (vi) implies (iv), assume that $f_1,\ldots,f_n\in R$ are $k$-linearly independent and let $\lambda_1,\ldots,\lambda_n\in k$ with $\lambda_1\s(f_1)+\cdots+\lambda_n\s(f_n)=0$. Let $\mu_i\in k^*$ with $\s(\mu_i)=\lambda_i$. Then
	$\s(\mu_1f_1+\cdots+\mu_nf_n)=0$. But $\mu_1f_1+\cdots+\mu_nf_n\in R^*$ and thus $\mu_1f_1+\cdots+\mu_nf_n=0$. So the $\mu_i$'s and therefore also the $\lambda_i$'s are zero.

	 Next we will show that (iv) implies (vii).
	  Assume that $f\in R\otimes_k S$ with $\s(f)=0$. Let $(f_i)$ be a $k$-basis of $R$. Then
	$f=\sum f_i\otimes s_i$ with $s_i\in S$ and $\sum \s(f_i)\otimes\s(s_i)=\s(f)=0$. Because the $\s(f_i)$'s are $k$-linearly independent we have $\s(s_i)=0$ for all $i$. But as $S$ is $\s$-reduced this implies that $s_i=0$ and therefore $f=0$.
	
	Finally, the implication (vi)$\Rightarrow$(i) is trivial.
\end{proof}

\begin{cor} \label{cor: ssep iff sep}
	Let $k$ be field of positive characteristic $p$ and let $q$ be a power of $p$. Let $R$ be a $k$-algebra and consider $k$ and $R$ as difference rings via the Frobenius endomorphism, i.e., $\s(f)=f^q$ for $f\in R$. Then $R$ is $\s$-separable over $k$ if and only if $R$ is a separable $k$-algebra.
\end{cor}
\begin{proof}
	Assume that $R$ is $\s$-separable over $k$ and let $K$ be a field extension of $k$. We have to show that $R\otimes_k K$ is reduced. But if we consider $K$ as a $\s$-field via $\s(f)=f^q$, then, by assumption, $\s$ is injective on $R\otimes_k K$. Thus $R\otimes_k K$ is reduced.
	
	Conversely, assume that $R$ is a separable $k$-algebra. By \cite[Theorem 2, Chapter V, \S 15, No. 4, A.V.123]{Bourbaki:Algebra2} the elements $f_1^p,\ldots,f_n^p$ are $k$-linearly independent if $f_1,\ldots,f_n\in R$ are $k$-linearly independent. Therefore $\s(f_1),\ldots,\s(f_n)$ are $k$-linearly independent and it follows from Proposition \ref{prop: sseparable} that $R$ is $\s$-separable over $k$.
\end{proof}

\begin{cor} \label{cor: tensor product of sseparable salgebras is sseparable}
	If $R$ and $S$ are $\s$-separable \ks-algebras, then also $R\otimes_k S$ is a \mbox{$\s$-s}eparable \ks-algebra.
\end{cor}
\begin{proof}
	Let $K$ be a $\s$-field extension of $k$. Then $S\otimes_k K$ is $\s$-reduced because $S$ is \mbox{$\s$-s}eparable. Therefore $(R\otimes_k S)\otimes_k K=R\otimes_k (S\otimes_k K)$ is $\s$-reduced by Proposition~\ref{prop: sseparable}.
\end{proof}

\begin{cor} \label{cor: sseparable invariant under base extension}
	Let $R$ be a \ks-algebra and $K$ a $\s$-field extension of $k$. Then $R$ is $\s$-separable over $k$ if and only if $R\otimes_k K$ is $\s$-separable over $K$.
\end{cor}
\begin{proof}
	Let us first assume that $R$ is $\s$-separable over $k$. Let $L$ be a $\s$-field extension of $K$. Then $(R\otimes_k K)_K L=R\otimes_k L$ is $\s$-reduced. Thus $R\otimes_k K$ is $\s$-separable over $K$. 
	
	Conversely, if $R\otimes_k K$ is $\s$-separable over $K$, it follows that $(R\otimes_k K)\otimes_K K^*=R\otimes_k K^*$ is $\s$-reduced. Thus $R$ is $\s$-separable over $k$ by Proposition \ref{prop: sseparable}.  
\end{proof}

\begin{cor} \label{cor: sfield extension of inversive is sseparable}
	Let $K|k$ be an extension of $\s$-fields. If $k$ is inversive, then $K$ is \mbox{$\s$-s}eparable over $k$.
\end{cor}
\begin{proof}
	By Proposition \ref{prop: sseparable} it suffices to note that a $\s$-field is $\s$-reduced.
\end{proof}

\subsection{Strongly $\sigma$-\'{e}tale $k$-$\s$-algebras}


In this subsection we introduce \ssetale{} \ks-algebras and establish their basic properties.

%

Recall (\cite[Chapter V, \S 6]{Bourbaki:Algebra2}) that an algebra $R$ over a field $k$ is called \'{e}tale if $R\otimes_k\overline{k}$ is isomorphic to a finite direct product of copies of the algebraic closure $\overline{k}$ of $k$. (In particular, $R$ is finite dimensional as a $k$-vector space.) See \cite[Theorem 4, Chapter V, \S 6]{Bourbaki:Algebra2} or \cite[Theorem 6.2]{Waterhouse:IntroductiontoAffineGroupSchemes} for 
a number of characterizations of \'{e}tale algebras. One of these characterizations is that $R$ is a finite direct product of finite separable field extensions of $k$, say $R=K_1\times\ldots\times K_n$. Then $e_1=(1,0,\ldots,0),\ e_2=(0,1,0,\ldots,0),\ldots, e_n=(0,\ldots,0,1)$ is the set of primitive idempotent elements of $R$. Moreover, the set $\{e_1,\ldots,e_n\}$ of primitive idempotent elements of $R$ is characterized by the following properties:
 \begin{enumerate}
 \item $e_1,\ldots,e_n$ are idempotent,
 \item $e_ie_j=0$ for $i\neq j$,
 \item $e_1+\ldots+e_n=1$,
 \item $n$ is maximal.
 \end{enumerate}


The following definition introduces a difference analog of \'{e}tale algebras.

\begin{defi} \label{defi: strongly setale}
	A $k$-$\s$-algebra is called \emph{\ssetale{}}\index{strongly $\s$-\'{e}tale \ks-algebra} if it is $\s$-separable over $k$ and \'{e}tale as a $k$-algebra.
\end{defi}

Note that by Proposition \ref{prop: sseparable}, over an inversive $\s$-field $k$, a \ks-algebra whose underlying $k$-algebra is \'{e}tale is \ssetale{} if and only if it is inversive. In particular, over an inversive $\s$-field, any 
$\s$-field extension that is finite and separable as a field extension is \ssetale{}. Let us give some more examples of \ssetale{} \ks-algebras. 

\begin{ex} \label{ex: ssetale}
	If $R=k\times\ldots\times k$ is an $n$-fold direct product of copies of  the $\s$-field $k$ and $\tau\in S_n$, we can extend $\s$ to $R$ by $\s((a_1,\ldots,a_n))=(\s(a_{\tau(1)}),\ldots,\s(a_{\tau_{(n)}}))$. Then $R$ is \ssetale{} because $R\otimes_k k^*$ is inversive (Cf. Proposition \ref{prop: sseparable}). On the other hand, if we define $\s$ on $R=k\times k$ by $\s((a_1,a_2))=(\s(a_1),\s(a_1))$, then $R$ is \'{e}tale as a $k$-algebra but not \ssetale{} because $\s$ is not injective on $R$.
\end{ex}

\begin{ex}\label{ex:special-mahler}	
	If we consider $k=\C(x)$, the field of rational functions over $\C$, as a $\s$-field via $\s(f(x))=f(x^2)$ and $R=k(x^{1/2})$ as a \ks-algebra via $\s(x^{1/2})=x$, then $R$ is \'{e}tale as a $k$-algebra but not \ssetale{} because $R$ is not $\s$-separable. To see this, note that $1$ and $x^{1/2}$ are $k$-linearly independent but $\s(1)=1$ and $\s(x^{1/2})=x$ are $k$-linearly dependent.
	
	On the other hand, $R=k(x^{1/3})$ with $\s(x^{1/3})=x^{2/3}$ is \ssetale{}. To see that $R$ is $\s$-separable over $k$, note that $1, x^{1/3}, x^{2/3}$ is a $k$-basis of $R$ such that also $\s(1),\s(x^{1/3}), \s(x^{2/3})$ is a $k$-basis, because the latter agrees with $1, x^{2/3}, xx^{1/3}$.
These considerations are generalized in Example~\ref{mahler-example}.
\end{ex}


\begin{lemma} \label{lemma: sfinite stable under base extension}
	Let $R$ be a $k$-$\s$-algebra and $K$ a $\s$-field extension of $k$. Then $R$ is a \ssetale{} $k$-$\s$-algebra if and only if $R\otimes_k K$ is a \ssetale{} $K$-$\s$-algebra.
\end{lemma}
\begin{proof}
	By \cite[Cor. 6.2, p. 47]{Waterhouse:IntroductiontoAffineGroupSchemes} or \cite[Cor. 2, Chapter V, \S 6, No. 5, A.V.32]{Bourbaki:Algebra2} the $k$-algebra $R$ is \'{e}tale if and only if the $K$-algebra $R\otimes_k K$ is \'{e}tale. Similarly, by Corollary \ref{cor: sseparable invariant under base extension} the $k$-$\s$-algebra $R$ is $\s$-separable if and only if the $K$-$\s$-algebra $R\otimes_k K$ is $\s$-separable.
\end{proof}

Let $R$ be a $\s$-ring. Following \cite{Levin:difference} an element $f\in R$ is called \emph{periodic} \index{periodic element} if $\s^n(f)=f$ for some $n\geq 1$.

\begin{lemma} \label{lemma: structure of sfinite ksalgebras}
	Let $R$ be a \ssetale{} $k$-$\s$-algebra. Then $\s$ induces a bijection on the set of primitive idempotent elements of $R$. Moreover, every idempotent element of $R$ is periodic.
\end{lemma}
\begin{proof}
	As before the beginning of Section \ref{subsec: sseparable salgebras}, let $\hs R=R\otimes_k k$ be the \ks-algebra obtained from $R$ by base extension via $\s\colon k\to k$.
	Because $R$ is $\s$-separable over $k$, it follows from Proposition \ref{prop: sseparable} that $$\psi\colon \hs R=R\otimes_k k\to R,\ f\otimes\lambda\mapsto \s(f)\lambda$$ is injective. As $R$ is a finite dimensional $k$-vector space, this implies that $\psi$ is an isomorphism of $k$-algebras. Let $e_1,\ldots,e_n$ denote the primitive idempotent elements of $R$. Because $\psi$ is an isomorphism, there are precisely $n$ primitive idempotent elements in $\hs R$. These must be $e_1\otimes 1,\ldots,e_n\otimes 1\in {\hs R}=R\otimes_k k$ because they satisfy the characteristic properties given before Definition \ref{defi: strongly setale}. So it follows that $\psi$ maps $\{e_1\otimes 1,\ldots, e_n\otimes 1\}$ bijectively onto $\{e_1,\ldots,e_n\}$.
 This shows that $e\mapsto \s(e)$ defines a bijection on $\{e_1,\ldots,e_n\}$. In particular, each $e_i$ is periodic. Since an arbitrary idempotent element of $R$ is a sum of certain $e_i$'s, it follows that any idempotent element of $R$ is periodic.
\end{proof}
%

\begin{lemma} \label{lemma: sgenerated by periodic idempotent implies ssetale}
	A \ks-algebra $\s$-generated by finitely many periodic idempotent elements is \ssetale{}.
\end{lemma}
\begin{proof}
	Let $R$ be \ks-algebra $\s$-generated by finitely many periodic idempotent elements. Because the $\s$-generators are periodic, $R$ is in fact generated as a $k$-algebra by finitely many periodic idempotent elements. But then $R$ is in fact generated by finitely many periodic idempotent elements as a $k$-vector space. In particular, $R$ is a finite dimensional $k$-vector space.
	
	We claim that as a $k$-algebra $R$ is isomorphic to a finite product of copies of $k$. To see this, note that for a non-trivial idempotent element $e\in R$ we have $k[e]\simeq k\times k$. Tensor products of $k$-algebras of this form are isomorphic to a finite product of copies of $k$. Also a quotient of a $k$-algebra isomorphic to a finite product of copies of $k$ is again isomorphic to a finite product of copies of $k$. Since $R$ can be written as a quotient of a finite tensor product of $k$-algebras of the form $k[e]\simeq k\times k$ it follows that $R$ is isomorphic to a finite product of copies of $k$.

%

  To show that $R$ is $\s$-separable over $k$ it suffices to show that
	$$\psi\colon \hs R=R\otimes_ k k\to R,\ r\otimes\lambda\mapsto\lambda\s(r)$$ is injective (Proposition \ref{prop: sseparable}). Let $e_1,\ldots,e_n$ be the primitive idempotent elements in $R\simeq k\times\ldots\times k$. These are a $k$-basis of $R$ and therefore $e_1\otimes 1,\ldots,e_n\otimes 1\in R\otimes_k k$ is a $k$-basis of $\hs R$.
	Since $R$ is isomorphic to $k\times\ldots\times k$ as a $k$-algebra, we see that also ${\hs R}$ is isomorphic to  $k\times\ldots\times k$ as a $k$-algebra. So the primitive idempotent elements in $\hs R$ are $e_1\otimes 1,\ldots,e_n\otimes1$. Therefore every idempotent element in $\hs R$ is of the form $e\otimes 1$ for some idempotent element $e\in R$. Every ideal in $\hs R$ is generated by an idempotent element. Thus the kernel of $\psi$ is generated by an element of the form $e\otimes 1$, where $e\in R$ is idempotent. So $\s(e)=0$.
	
	There exists a $k$-basis $d_1,\ldots,d_n$ of $R$ consisting of periodic idempotent elements. We can find an integer $m\geq 1$ such that $\s^m(d_i)=d_i$ for $i=1,\ldots,n$. We may write $e=\lambda_1 d_1+\cdots+\lambda_n d_n$ for $\lambda_1,\ldots,\lambda_n\in k$. Then
	$$0=\s^m(e)=\s^m(\lambda_1)d_1+\cdots+\s^m(\lambda_n)d_n.$$
	Therefore $\s^m(\lambda_1)=\cdots=\s^m(\lambda_n)=0$ and consequently also $e=0$. This shows that $\psi$ is injective.
\end{proof}

It is sometimes convenient to be able to assume that the base $\s$-field $k$ is algebraically closed or separably algebraically closed. This can be achieved by passing to the (separable) algebraic closure. In this respect the following remark is relevant.

\begin{rem} \label{rem: extend s}
	Let $k$ be a $\s$-field and let $\overline{k}$ denote an algebraic closure of $k$. Then $\s\colon k\to k$ can be extended to an endomorphism $\s\colon \overline{k}\to\overline{k}$. The separable algebraic closure $k_s\subset\overline{k}$ of $k$ is stable under $\s$. In particular, $\s\colon k\to k$ can be extended to an endomorphism on the separable algebraic closure of $k$.
\end{rem}
\begin{proof}
	By \cite[Corollary 5.1.16]{Levin:difference} there exists an extension $\s\colon\overline{k}\to\overline{k}$ of $\s\colon k\to k$.

	Recall, for a polynomial $f$ with coefficients in $k$, we write ${\hs f}$ for the polynomial obtained by applying $\s$ to the coefficients of $f$. If $1$ lies in the ideal generated by $f$ and its formal derivative, then $1$ also lies in the ideal generated by ${\hs f}$ and its formal derivative. Thus, ${\hs f}$ is separable if $f$ is separable. If $a\in\overline{k}$ is a solution of $f$, then $\s(a)$ is a solution of ${\hs f}$. Therefore $\s(k_s)\subset k_s$.
\end{proof}

The following proposition provides a partial converse to the last lemma.

\begin{prop} \label{prop: characterize sfinite}
	The $R$ be a $k$-$\s$-algebra that is finite dimensional as a $k$-vector space. Let $k_s$ denote the separable algebraic closure of $k$, equipped with an extension of $\s\colon k\to k$. Then $R$ is \ssetale{} if and only if $R\otimes_k k_s$ is generated as a $k_s$-vector space by periodic idempotent elements.
\end{prop}
\begin{proof}
	By Lemma \ref{lemma: sfinite stable under base extension} we may assume $k=k_s$.
	
	Let us first assume that $R$ is \ssetale{}. Since $k=k_s$ and $R$ is an \'{e}tale $k$-algebra, $R$ is generated as a $k$-vector space by idempotent elements (cf. \cite[Theorem 6.2]{Waterhouse:IntroductiontoAffineGroupSchemes}). By Lemma \ref{lemma: structure of sfinite ksalgebras} every idempotent element of $R$ is periodic.
	
	Conversely, if $R$ is generated by periodic idempotent elements, then it follows from Lemma~\ref{lemma: sgenerated by periodic idempotent implies ssetale} that $R$ is \ssetale{}.
	%
	%
\end{proof}

\begin{lemma} \label{lemma: sfinite algebras}
	\begin{enumerate}
		\item A $k$-$\s$-subalgebra of a \ssetale{} $k$-$\s$-algebra is \ssetale{}.
	\item The tensor product of two \ssetale{} $k$-$\s$-algebras is \ssetale{}.
	\item The quotient of a \ssetale{} $k$-$\s$-algebra by a $\s$-ideal is a \ssetale{} $k$-$\s$-algebra.
\end{enumerate}
\end{lemma}
\begin{proof}
	The first statement follows from the fact that a $k$-subalgebra of an \'{e}tale $k$-algebra is \'{e}tale (\cite[Prop. 3, Chapter V, \S 6, No. 4, A.V.30]{Bourbaki:Algebra2}) and the obvious fact that a \ks-subalgebra of a $\s$-separable \ks-algebra is $\s$-separable.
	
	The second statement follows from the fact that the tensor product of two \'{e}tale \mbox{$k$-a}lgebras is \'{e}tale (\cite[Cor. 1, Chapter V, \S 6, No. 5, A.V.32]{Bourbaki:Algebra2}) and the fact that the tensor product of two $\s$-separable $k$-$\s$-algebras is $\s$-separable (Corollary \ref{cor: tensor product of sseparable salgebras is sseparable}).
	
	Finally we prove the third statement. Let $R$ be a \ssetale{} $k$-$\s$-algebra and $\ida$ a $\s$-ideal of $R$. Let $k_s$ denote the separable algebraic closure of $k$, equipped with an extension of $\s$, as per Remark \ref{rem: extend s}. Then
	$$(R/\ida)\otimes_k k_s=(R\otimes_k k_s)/(\ida\otimes_k k_s).$$
	It follows from Proposition \ref{prop: characterize sfinite}, that the right-hand side is generated as a $k_s$-vector space by periodic idempotent elements. Applying the proposition again yields the conclusion.
\end{proof}

The following lemma shows that being \ssetale{} is transitive.
\begin{lemma} \label{lemma: transitivity of ssetale}
	Let $K|k$ be a \ssetale{} $\s$-field extension, i.e., $K$ is a finite, separable, $\s$-separable $\s$-field extension of $k$. Let $R$ be a \ssetale{} $K$-$\s$-algebra. Then $R$ is a \ssetale{} \ks-algebra. 
\end{lemma}
\begin{proof}
	Since $R$ is an \'{e}tale $K$-algebra and $K$ is an \'{e}tale $k$-algebra it follows from \cite[Cor. 2 (b), A.V.32]{Bourbaki:Algebra2} that $R$ is an \'{e}tale $k$-algebra. Thus it suffices to show that $R$ is $\s$-separable over $k$. So we have to show that the map
	\begin{align*}\psi\colon {\hs R}=R\otimes_k k & \to R \\ r\otimes\lambda & \mapsto \s(r)\lambda	
	\end{align*}
	is injective. Because $R$ is a finite dimensional $k$-vector space and $\psi$ is $k$-linear, it suffices to show that $\psi$ is surjective. So let $r\in R$.
	By assumption the canonical maps
	$R\otimes_K K\to R$ and $K\otimes_k k\to K$ are injective and thus surjective. So we can write $r=\sum \s(r_i)\mu_i$ with $r_i\in R$ and $\mu_i\in K$. Similarly, we can write $\mu_i=\sum\s(a_{ij})\lambda_{ij}$ with $a_{ij}\in K$ and $\lambda_{ij}\in k$. Then
	$$r=\sum_{i,j}\s(r_ia_{ij})\lambda_{ij}$$ lies in the image of $\psi$.	
\end{proof}


\subsection{The strong core}



\begin{defi} \label{defi:  pis}
	Let $R$ be a $k$-$\s$-algebra. We define the \emph{strong core}
	$$\pis(R)=\pis(R|k)$$ as the union of all \ssetale{} $k$-$\s$-subalgebras of $R$.
\end{defi}
In this section we show that this construction has good functorial properties.
\begin{rem}
We have that $\pis(R)$ is a $\s$-separable $k$-$\s$-subalgebra of $R$. 
Indeed, if $R_1$ and $R_2$ are \ssetale{} $k$-$\s$-subalgebras of a $k$-$\s$-algebra $R$, then also $R_1R_2$ is a \ssetale{} $k$-$\s$-algebra by Lemma \ref{lemma: sfinite algebras}, since $R_1R_2$ may be written as a quotient of $R_1\otimes_k R_2$. 
\end{rem}

 Note that $\pis(R)$ need not be \ssetale{} in general. Indeed, $\pis(R)$ is strongly $\s$-\'{e}tale if and only if $\pis(R)$ is finitely $\s$-generated over $k$. This is also equivalent to $\pis(R)$ being finite dimensional as a $k$-vector space.
 

\begin{conj}\label{fin-si-gen}
If $R$ is a finitely $\s$-generated $k$-$\s$-algebra, then $\pis(R)$ is a \ssetale{} $k$-$\s$-algebra. 
\end{conj} 
  We will see later (Theorem~\ref{ssetale-hopf}) that $\pis(R)$ is a \ssetale{} \ks-subalgebra of $R$ if $R$ is a finitely $\s$-generated $k$-$\s$-Hopf algebra. Before establishing further properties of the strong core, let us look at an example. Another example computation of a strong core is given in Example~\ref{mahler-example}.

\begin{ex}
	Let $R=k\times\ldots\times k$ be an $n$-fold direct product of copies of $k$ and let $\tau\colon\{1,\ldots,n\}\to \{1,\ldots,n\}$ be a (not necessarily bijective) map. We extend $\s$ to $R$ by $$\s((a_1,\ldots,a_n))=(\s(a_{\tau(1)}),\ldots,\s(a_{\tau(n)})).$$ We will determine $\pis(R)$.	
		We begin by defining an equivalence relation on $\{1,\ldots,n\}$ by $i\sim j$ if there exists an $n\geq 0$ such that $\tau^n(i)=\tau^n(j)$.
		Set $$S=\{(a_1,\ldots,a_n)\in R\,|\ a_i=a_j \text{ if } i\sim j\}.$$
	As $i\sim j$ implies $\tau(i)\sim\tau(j)$, 
	we see that $S$ is a \ks-subalgebra of $R$. Moreover, $\tau$ induces a bijection on the equivalence classes. It follows that $S$ is isomorphic to a \ssetale{} \ks-algebra of the form described in the beginning of Example \ref{ex: ssetale}. In particular, $S\subset \pis(R)$. 
	 
	 Note that $R\otimes_k k^*=k^*\times\ldots\times k^*$ and $\s$ on $R\otimes_k k^*$ is given by exactly the same formula as on $R$. Furthermore, $\s$ is bijective on $S\otimes_k k^*\subset R\otimes_k k^*$.
	 
	 Suppose that $S_1$ is a \ks-subalgebra of $R$ that properly contains $S$. If $f$ is any element of $R$, then $\s^m(f)\in S$ for $m$ large enough. In particular, given $f\in S_1\smallsetminus S$, we have $\s^m(f\otimes 1)\in S\otimes_k k^*$ for some $m$. Because $\s$ is bijective on $S\otimes_k k^*$, we see that $\s^m(f\otimes 1)=\s(h)$ for  some $h\in S\otimes_k k^*$. This shows that $\s$ is not injective on $S_1\otimes_k k^*$. Thus $S_1$ is not $\s$-separable over $k$. We conclude that $\pis(R)=S$.	
\end{ex}

\begin{lemma} \label{lemma: pis for separably algebraically closed}
	Let $R$ be a $k$-$\s$-algebra. If $k$ is separably algebraically closed, then $\pis(R)$ is the $k$-subvector space of $R$ generated by all periodic idempotent elements of $R$.
\end{lemma}
\begin{proof}
	It is clear from Proposition \ref{prop: characterize sfinite} that $\pis(R)$ is contained in the $k$-subvector space of $R$ generated by all periodic idempotent elements of $R$.
	
	On the other hand, $k\{e\}\subset\pis(R)$ for every periodic idempotent element $e\in R$ by Lemma \ref{lemma: sgenerated by periodic idempotent implies ssetale}.
\end{proof}


We will need a simple algebraic (rather than difference algebraic) lemma on idempotent elements.  

\begin{lemma} \label{lemma: idempotents descend}
	Let $k$ be an algebraically closed field, $K$ a field extension of $k$ and $R$ a $k$-algebra. Then every idempotent element $e$ of $R\otimes_k K$ is of the form $e=f\otimes 1$ for some idempotent element $f\in R$.
\end{lemma}
\begin{proof}
	A field extension of an algebraically closed field is geometrically connected (in the sense of Definition \cite[Tag 037T]{stacks-project}). The claim thus follows from \cite[Tag 037W]{stacks-project}.
\end{proof}

Our next important goal is to show that the formation of $\pis$ is compatible with extension of the base $\s$-field. To prove this for the extension $k_s|k$, where $k_s$ denotes the separable algebraic closure of $k$ we will use Galois descent. In this connection we need to introduce an action of $\s$ on the absolute Galois group of $k$.
To define this action we need a simple lemma on field extensions. We include the proof for the sake of completeness. The reader should be warned that in this article Galois extensions are not necessarily finite.


\begin{lemma} \label{lemma: extension of s almost unique}
	Let $K$ be a $\s$-field, $L$ a Galois extension of $k$ and $\s_1,\s_2\colon L\to L$ two extensions of $\s\colon K\to K$. Then there exists a unique element $\tau$ in the Galois group of $L|K$ such that $\s_2=\s_1\tau$.
\end{lemma}
\begin{proof}
	The uniqueness of $\tau$ follows from the injectivity of $\s_1$.
	Let $\mathcal{M}$ denote the set of all pairs $(M,\tau)$ where $M$ is an intermediate field of $L|K$ and $\tau\colon M\to L$ a morphism of field extensions of $K$ satisfying $\s_2(a)=\s_1(\tau(a))$ for all $a\in M$. By Zorn's Lemma there exists a maximal element $(M,\tau)$ in $\mathcal{M}$. Suppose that $M$ is properly contained in $L$. Choose $a\in L\smallsetminus M$ and let $f$ denote the minimal polynomial of $a$ over $M$. Let ${}^{\tau}\!f$ denote the polynomial obtained from $f$ by applying $\tau$ to the coefficients. As $L$ is Galois over $K$, we see that ${}^{\tau}\!f$ cannot have a root outside $L$. Thus ${}^{\tau}\!f$ has all its roots in $L$. Since ${}^{\s_2}\!f={}^{\s_1\tau}\!f$ we see that $\s_1$ maps the roots of ${}^{\tau}\!f$ onto the roots of ${}^{\s_2}\!f$. As $\s_2(a)\in L$ is a root of ${}^{\s_2}\!f$ this shows that there exists a root $b\in L$ of ${}^{\tau}\!f$ such that $\s_1(b)=\s_2(a)$. So we can extend $\tau$ to $\tau\colon M(a)\to L$ by setting $\tau(a)=b$; a contradiction.
\end{proof}

Let $k$ be a $\s$-field and let $k_s$ denote the separable algebraic closure of $k$. Let $\Gal$ denote the Galois group of $k_s|k$. As noted in Remark \ref{rem: extend s} there exists an extension of $\s\colon k\to k$ to an endomorphism of $k_s$. Let us fix such an extension. For $\tau\in\Gal$ the maps $\s$ and $\tau\s$ are extensions of $\s\colon k\to k$ to $k_s$. By Lemma \ref{lemma: extension of s almost unique} there exists a unique $\tau^\s\in\Gal$ such that
\begin{equation} \label{eq: operation of gal}
\tau\s=\s\tau^\s.
\end{equation}
Then
$$\s\colon \Gal\to\Gal,\ \tau\mapsto \tau^\s$$
is a morphism of groups.

If $k$ and therefore also $k_s$ is inversive, this action of $\s$ on the Galois group agrees with the action employed in \cite[Section 8.1]{Levin:difference}, where the formula $\tau^\s=\s^{-1}\tau\s$ is used as the definition. Note, however, that this formula does not make sense in our context because $\s$ need not be invertible.

The formation of the strong core is compatible with extension of the base $\s$-field:

\begin{lemma} \label{lemma: pis stable under base extension}
	Let $R$ be a $k$-$\s$-algebra and $K$ a $\s$-field extension $k$. Then
	$$\pis(R\otimes_k K|K)=\pis(R|k)\otimes_k K.$$
\end{lemma}
\begin{proof}
	The inclusion ``$\supset$'' is clear from Lemma \ref{lemma: sfinite stable under base extension}. We will prove the equality in several steps.
	
	\medskip
	
	Step 1: Assume $K=k_s$ is the separable algebraic closure of $k$ equipped with an extension of $\s\colon k\to k$. Let $\Gal$ denote the Galois group of $k_s|k$. As explained above there is an action of $\s$ on $\Gal$ denoted by $\tau\mapsto \tau^\s$. There also is a natural action of $\Gal$ on $R\otimes_k k_s$ via the right factor: $\tau(r\otimes\lambda)=r\otimes\tau(\lambda)$ for $\tau\in\Gal$, $r\in R$ and $\lambda\in k$.	For $f\in R\otimes_k k_s$ and $\tau\in\Gal$ we have $\tau(\s(f))=\s(\tau^\s(f))$.
	
	
	We will show that $\pis(R\otimes_k k_s|k_s)\subset R\otimes_k k_s$ is stable under the $\Gal$-action. We know from Lemma \ref{lemma: pis for separably algebraically closed} that $\pis(R\otimes_k k_s|k_s)$ is generated as a $k_s$-vector space by the periodic idempotent elements of $R\otimes_k k_s$.
	It therefore suffices to show that $\tau(e)$ is periodic and idempotent for every periodic idempotent element $e\in R\otimes_k k_s$ and every $\tau\in\Gal$. As $\tau$ defines an automorphism of $R\otimes_k k_s$, clearly also $\tau(e)$ is idempotent. Assume $\s^n(e)=e$. For $i\in\mathbb{N}$ we have
	\begin{equation} \label{eqn: tau(e) periodic}
	\tau(e)=\tau(\s^{ni}(e))=\s^{ni}(\tau^{\s^{ni}}(e)).
	\end{equation}	
	Since every orbit of $\Gal$ on $k_s$ is finite, also every orbit of $\Gal$ on $R\otimes_k k_s$ is finite. Therefore
	$\{\tau(e), \tau^{\s^{n}}(e), \tau^{\s^{2n}}(e),\ldots\}$ is a finite set. It follows that there exist $i<j$ such that $\tau^{\s^{ni}}(e)=\tau^{\s^{nj}}(e)$. Then
	$$\tau(e)=\s^{ni}(\tau^{\s^{ni}}(e))=\s^{ni}(\tau^{\s^{nj}}(e))=\tau^{\s^{n(j-i)}}(\s^{ni}(e))=\tau^{\s^{n(j-i)}}(e).$$
	Thus, replacing $i$ with $j-i$ in equation (\ref{eqn: tau(e) periodic}), we obtain $\tau(e)=\s^{n(j-i)}(\tau(e))$. Therefore $\tau(e)$ is periodic as required.
	
	So $\pis(R\otimes_k k_s|k_s)\subset R\otimes_k k_s$ is stable under the $\Gal$-action and it follows from Galois descent (\cite[Prop. 6, Chapter V, \S 10, No. 4, A.V.62]{Bourbaki:Algebra2} or \cite[Prop. 11.1.4, p. 186]{Springer:LinearAlgebraicGroups}) that $\pis(R\otimes_k k_s|k_s)=S\otimes_k k_s$, where $S$ is the \ks-subalgebra of $R$ given by $S=R\cap\pis(R\otimes_k k_s|k_s)$. It now suffices to show that $k\{s\}$ is \ssetale{} for every $s\in S$. But a finitely $\s$-generated $k_s$-$\s$-subalgebra of $\pis(R\otimes_k k_s|k_s)$ is \ssetale{}, so $k_s\{s\}=k\{s\}\otimes_k k_s$ is \ssetale{} and it follows from Lemma \ref{lemma: sfinite stable under base extension} that $k\{s\}$ is \ssetale{}.
	
	\medskip
	
	Step 2: Assume that $k=k_s$ is separably algebraically closed and that $K=\overline{k}$ is the algebraic closure of $k$. If the characteristic of $k$ is zero, $k_s=\overline{k}$ and there is nothing to prove. So we may assume that the characteristic of $k$ is $p>0$. Then $\overline{k}$ is purely inseparable over $k$.
	By Lemma \ref{lemma: pis for separably algebraically closed} it suffices to show that every periodic idempotent element $e=\sum r_i\otimes\lambda_i\in R\otimes_k\overline{k}$ belongs to $R$. As $\overline{k}$ is purely inseparable over $k$, there exists an integer $n\geq 1$ such that $\lambda_i^{p^n}\in k$ for all $i$. So $e=e^{p^n}=\sum r_i^{p^n}\otimes\lambda_i^{p^n}$ lies in $R$.
	
	\medskip
	
	Step 3: Assume that $K=\overline{k}$ is the algebraic closure of $k$. 
	Then the separable algebraic closure $k_s$ of $k$ is naturally a $\s$-subfield of $\overline{k}$ (Remark \ref{rem: extend s}) and it follows from steps 2 and 1 that	
	\begin{align*}
	\pis(R\otimes_k \overline{k}|\overline{k}) & =\pis((R\otimes_k k_s)\otimes_{k_s}\overline{k}|\overline{k})=\pis(R\otimes_k k_s|k_s)\otimes_{k_s}\overline{k}=\pis(R|k)\otimes_k k_s\otimes_{k_s}\overline{k}= \\ & =\pis(R|k)\otimes_k\overline{k}.
	\end{align*}
	
	\medskip
	
	Step 4: Assume that $k$ and $K$ are algebraically closed. We know from Lemma \ref{lemma: pis for separably algebraically closed} that $\pis(R\otimes_k K|K)$ is generated as a $K$-vector space by periodic idempotent elements. By Lemma \ref{lemma: idempotents descend} every idempotent element $e$ of
	$R\otimes_k K$ is of the form $e=f\otimes 1$ for some idempotent element $f\in R$. If $e$ is periodic, then also $f$ is periodic. Therefore $\pis(R\otimes_k K|K)$ is generated as a $K$-vector space by $\pis(R|k)$ and so $\pis(R\otimes_k K|K)=\pis(R|k)\otimes_k K.$
	
	\medskip
	
	Step 5: Finally we treat the general case, i.e., $K$ is an arbitrary $\s$-field extension of $k$. Let $\overline{K}$ denote the algebraic closure of $K$ and choose an extension of $\s\colon K\to K$ to $\overline{K}$. Then the algebraic closure $\overline{k}$ of $k$ is naturally a $\s$-subfield of $\overline{K}$. As $R\otimes_k\overline{K}=(R\otimes_k\overline{k})\otimes_{\overline{k}}\overline{K}$ it follows from steps 4 and 3 that
	$$\pis(R\otimes_k\overline{K}|\overline{K})=\pis(R\otimes_k\overline{k}|\overline{k})\otimes_{\overline{k}}\overline{K}=\pis(R|k)\otimes_k\overline{k}\otimes_{\overline{k}}\overline{K}=\pis(R|k)\otimes_k K\otimes_K \overline{K}.$$
	Moreover,
	\begin{align*}
	\pis(R|k)\otimes_k K\otimes_K\overline{K} & \subset\pis(R\otimes_k K|K)\otimes_K\overline{K}\subset\pis(R\otimes_k K\otimes_K\overline{K}|\overline{K})= \\
	& =\pis(R\otimes_k\overline{K}|\overline{K}) =\pis(R|k)\otimes_k K\otimes_K \overline{K}.
	\end{align*}
	Therefore $(\pis(R)\otimes_k K)\otimes_K\overline{K}=\pis(R\otimes_k K)\otimes_K\overline{K}$ and consequently $\pis(R|k)\otimes_k K=\pis(R\otimes_k K|K)$ as desired.
\end{proof}

The following lemma shows that the formation of $\pis(R)$ is compatible with the tensor product.

\begin{lemma} \label{lemma: pis compatible with tensor product}
	Let $R$ and $S$ be $k$-$\s$-algebras. Then
	$$\pis(R\otimes_k S)=\pis(R)\otimes_k\pis(S).$$
\end{lemma}
\begin{proof}
	The inclusion ``$\supset$'' follows from the fact that the tensor product of two \ssetale{} $k$-$\s$-algebras is \ssetale{} (Lemma \ref{lemma: sfinite algebras}).
	
	To prove the inclusion ``$\subset$'' we first reduce to the case that $k$ is algebraically closed and inversive. Let $K$ be an algebraically closed and inversive $\s$-field extension of $k$. For example, $K$ could be the algebraic closure of the inversive closure of $k$.
	We assume that the lemma is true over $K$.
	So $$\pis((R\otimes_k K)\otimes_K (S\otimes_k K)|K)=\pis(R\otimes_k K|K)\otimes_K\pis(S\otimes_k K|K).$$ As $(R\otimes_k K)\otimes_K (S\otimes_k K)=(R\otimes_k S)\otimes_k K$, we see that the left hand side equals $\pis(R\otimes_k S)\otimes_k K$ by Lemma \ref{lemma: pis stable under base extension}. Similarly, also by Lemma \ref{lemma: pis stable under base extension}, the right hand side agrees with $(\pis(R)\otimes_k \pis(S))\otimes_k K$. It follows that $\pis(R\otimes_k S)=\pis(R)\otimes_k\pis(S)$. We can therefore assume that $k$ is algebraically closed and inversive.
	
	 We know from Lemma \ref{lemma: pis for separably algebraically closed} that $\pis(R\otimes_k S)$ is generated as a $k$-vector space by periodic idempotent elements. So let $e\in R\otimes_k S$ be a periodic idempotent element. It suffices to show that $e\in\pis(R)\otimes_k \pis(S)$.
	
	For a $k$-algebra $A$, let $\pi_0(A)$ denote the union of all \'{e}tale $k$-subalgebras of $A$. It follows from \cite[Section 6.5]{Waterhouse:IntroductiontoAffineGroupSchemes} that $\pi_0(A\otimes_k B)=\pi_0(A)\otimes_k \pi_0(B)$ for $k$-algebras $A$ and $B$.
	
	So $e\in\pi_0(R\otimes_k S)=\pi_0(R)\otimes_k \pi_0(S)$. Thus there exist \'{e}tale $k$-subalgebras $R_1$ and $S_1$ of $R$ and $S$ respectively such that $e\in R_1\otimes_k S_1$.
	
	If $V_i$ is a $k$-subspace of $R$ such that $e\in V_i\otimes_k S$, then $e\in\cap_{i}(V_i\otimes_k S)=(\cap_{i}V_i)\otimes_k S$. Thus there exists a smallest $k$-subspace $V$ of $R$ such that $e\in V\otimes_k S$. Let $v_1,\ldots,v_n$ be a $k$-basis of $V$ and write $e=\sum v_i\otimes s_i\in V\otimes_k S$. The idempotent element $e$ is periodic, say $\s^m(e)=e$. So $e=\sum \s^m(v_i)\otimes\s^m(s_i)$. By construction of $V$, the $k$-subspace of $R$ generated by $\s^m(v_1),\ldots,\s^m(v_n)$ contains $V$. Since $V$ has dimension $n$, this shows that the $\s^m(v_i)$ are a $k$-basis of $V$. As $k$ is inversive, we see that $\s^m(V)=V$. Let $W=V+\s(V)+\cdots+\s^{m-1}(V)$. Then $W$ is a $k$-subspace of $R$ and $\s(W)=W$. Indeed, for dimension reasons, $\s$ induces a bijection on $W$.
	
	As $V$ is contained in $R_1$, every element $v$ of $V$ satisfies a separable polynomial over $k$. Thus also $\s(v)$ satisfies a separable polynomial over $k$. (Cf. the proof of Remark \ref{rem: extend s}). It follows that all the elements of $W$ satisfy separable polynomials over $k$.
	
	If $f$ is a separable polynomial over $k$, then $k[x]/(f)$ is an \'{e}tale algebra. Since tensor products and quotients of \'{e}tale algebras are \'{e}tale algebras (\cite[Corollary 6.2]{Waterhouse:IntroductiontoAffineGroupSchemes}), it follows that any $k$-algebra that can be generated by finitely many elements satisfying separable polynomials over $k$, is \'{e}tale. Therefore $k\{W\}=k[W]\subset R$ is an \'{e}tale $k$-algebra.

	 Because $k$ is inversive and $\s$ induces a bijection of $W$, we see that $\s$ is a bijection on $k\{W\}=k[W]$. Thus $k\{W\}$ is a \ssetale{} \ks-algebra. Therefore $e\in k\{W\}\otimes_k S\subset \pis(R)\otimes_k S$.
	
	A similar argument shows that $e\in R\otimes_k \pis(S)$ and it follows (using \cite[Prop. 7, Chapter I, \S 2.6, p. 18]{Bourbaki:commutativealgebra}) that
	$$e\in (\pis(R)\otimes_k S)\cap (R\otimes_k \pis(S))=\pis(R)\otimes_k\pis(S).$$

\end{proof}

\begin{cor}\label{monoidal}
The assignment $R\rightsquigarrow\pis(R)$ gives rise to a monoidal functor on the category of \ks-algebras. 
\end{cor}
\begin{proof}
If $\psi\colon R\to S$ is a morphism of \ks-algebras, it follows from Lemma \ref{lemma: sfinite algebras} that $\psi(\pis(R))\subset\pis(S)$.
The compatibility with the tensor product is guaranteed by Lemma \ref{lemma: pis compatible with tensor product}.
\end{proof}


\subsection{Strong core versus core}

Note that $\pis(L|K)$ is a $\s$-field for an extension of $\s$-fields $L|K$. Our first goal is to compare $\pis(L|K)$ to the core of $L|K$.
 Recall (\cite[Def. 4.3.17]{Levin:difference}) that the core
$$\core(L|K)$$ of an extension of $\s$-fields $L|K$ is the union of all intermediate $\s$-fields of $L|K$, whose underlying field is a finite separable field extension of $K$. The core $\core(L|K)$ is a $\s$-subfield of $L$ that plays a fairly prominent role in classical difference algebra, for example in the study of compatibility of $\s$-field extensions (\cite[Theorem 5.4.22]{Levin:difference}) or in questions related to extending specializations (\cite[Theorem XI, Chapter 7]{Cohn:difference}).

\begin{rem} \label{rem: strong core vs core}
Let $L|K$ be a $\s$-field extension. 
\begin{enumerate}
\item We have that $\pis(L|K)\subset\core(L|K)$.
\item If $L$ is finitely $\s$-generated over $K$, $\pis(L|K)$ is a finite separable field extension of $K$.
\item We have that $\pis(L|K)=\core(L|K)$ if and only if $\core(L|K)$ is $\s$-separable over $K$.
\item In particular, $\pis(L|K)=\core(L|K)$ if $K$ is inversive (Corollary \ref{cor: sfield extension of inversive is sseparable}).
\end{enumerate}
\end{rem}


\begin{ex}\label{mahler-example}
We will compute the core and the strong core for the $\s$-field extension $\bar{K}|K$, where $K=\bar{\mathbb{Q}}(x)$ with the Mahler difference operator $\sigma(f(x))=f(x^d)$ for some integer $d\geq 2$. As proved in \cite[Proposition~15]{julien-mahler}, the only finite difference field extensions of $K$ are the fields $K(x^{1/n})$, for $n\geq 1$. 

We claim that $K(x^{1/n})$ is $\sigma$-separable over $K$ if and only if $n$ is coprime to $d$. Indeed, note that $\sigma$ transforms the basis $1, x^{1/n},\ldots,x^{(n-1)/n}$ into $1, x^{d/n},...,x^{d(n-1)/n}$. 
If $n$ is coprime to $d$, then $d$ is invertible modulo $n$, and the second sequence is, up to a multiplication by some integer powers of $x$, the same as the first, and thus it is again a basis.
Conversely, suppose that $n$ and $d$ have a common divisor $l>1$, and write $n=ln'$, $d=ld'$. Then $x^{n'd/n}=x^{d'}$, so it is linearly dependent with $1$ over $K$, and the second sequence is not a basis.

We conclude that 
$$
\core(\bar{K}|K)=\bigcup_{n\geq 1}K(x^{1/n}), \ \ \ \ \text{and} \ \ \ \ \ \pis(\bar{K}|K)=\bigcup_{(n,d)=1}K(x^{1/n}).
$$
\end{ex}

Note that in the above example, for every $a\in\core(\bar{K}|K)$, there exists an $m\geq 0$ such that $\s^m(a)\in \pis(\bar{K}|K)$. We will show that indeed for any $\s$-field extension $L|K$ and $a\in\core(L|K)$, there exists an $m\in\N$ such that $\s^m(a)\in\pis(L|K)$. Following \cite[Def. 4.30]{Hrushovski:elementarytheoryoffrobenius} we make the following definition.

\begin{defi}
	An extension of $\s$-fields $L|K$ is called \emph{$\s$-radicial} \index{$\s$-radicial} if for every $a\in L$ there exists an $n\in \N$ such that $\s^n(a)\in K$.
\end{defi}

For example, the inversive closure $K^*$ (\cite[Def. 2.1.6]{Levin:difference}) of a $\s$-field $K$ is a $\s$-radicial extension of $K$. Moreover, every $\s$-radicial extension of $K$ is contained in $K^*$. For later reference we note the following obvious lemma.

\begin{lemma} \label{lemma: sradicial transitive}
	Let $L|K$ and $M|L$ be $\s$-radicial $\s$-field extensions. Then $M|K$ is also $\s$-radicial. \qed
\end{lemma} 

\begin{lemma} \label{lemma: finite implies sradicial over pis}
	Let $L|K$ be a finite separable extension of $\s$-fields. Then $L$ is $\s$-radicial over $\pis(L|K)$.
\end{lemma}
\begin{proof}
	The sequence of intermediate $\s$-fields $K\s^i(L)$ of $L|K$ is decreasing and therefore must stabilize, since $L|K$ is finite. So there exists an $n\in\N$ such that $K\s^{n+1}(L)=K\s^n(L)$.
	The canonical map
	\begin{align*}{\hs(K\s^n(L))}=K\s^n(L)\otimes_K K &\longrightarrow K\s^{n+1}(L)=K\s^n(L) \\
	a\otimes\lambda &\longmapsto \s(a)\lambda
	\end{align*}
	is surjective and $K$-linear. Thus it is injective and it follows from Proposition \ref{prop: sseparable} that $K\s^n(L)$ is $\s$-separable over $K$. Because $L$ is separable over $K$, also $K\s^n(L)$ is separable over $K$ and we find that $K\s^n(L)\subset\pis(L|K)$. Clearly $L$ is $\s$-radicial over $K\s^n(L)$. Thus $L$ is $\s$-radicial over $\pis(L|K)$.	
\end{proof}

\begin{cor} \label{cor: pis sradicial over core}
	Let $L|K$ be an extension of $\s$-fields. Then $\core(L|K)$ is $\s$-radicial over $\pis(L|K)$.
\end{cor}
\begin{proof}
	Let $a\in\core(L|K)$. Then $K\l a\r$ is a finite separable $\s$-field extension of $K$. It follows from Lemma \ref{lemma: finite implies sradicial over pis} that there exists an $n\in\N$ such that $\s^n(a)\in\pis(K\l a\r|K)\subset\pis(L|K)$.
\end{proof}

The following lemma will be needed in the proof of our enhanced version of Babbitt's decompositon (Theorem \ref{theo: Babbitt}). 

\begin{lemma} \label{lemma: pispis equals pis}
	Let $L|K$ be an extension of $\s$-fields. Then
	$$\pis(L|\pis(L|K))=\pis(L|K).$$
\end{lemma}
\begin{proof}
	We abbreviate $N=\pis(L|K)$. Let $M\subset L$ be a \ssetale{} $N$-$\s$-algebra. We have to show that $M\subset N$. By the primitive element theorem, there exists $a\in M$ such that $M=N(a)$. It suffices to show that $a\in N=\pis(L|K)$. Let $f$ denote the minimal polynomial of $a$ over $N$ and let $g$ denote a polynomial over $N$ such that $\s(a)=g(a)$. Let $N'$ be the $\s$-field extension $\s$-generated over $K$ by the coefficients of $f$ and $g$. Then $N'$ is \ssetale{} over $K$, $N'(a)$ is a $\s$-field and $f$ is also the minimal polynomial of $a$ over $N'$. We claim that $N'(a)$ is \ssetale{} over $N'$. Clearly $N'(a)$ is \'{e}tale over $N'$. So it suffices to show that
	${\hs (N'(a))}\to N'(a)$ is injective.
	Assume that $f$ has degree $n$. Then it suffices to show that $1,\s(a),\ldots,\s(a^{n-1})\in L$ are $N'$-linearly independent. But this is clearly the case, indeed $1,\s(a),\ldots,\s(a^{n-1})$ are $N$-linearly independent since $N(a)$ is $\s$-separable over $N$.
	
	So $N'(a)$ is \ssetale{} over $N'$ and $N'$ is \ssetale{} over $K$. It follows from Lemma \ref{lemma: transitivity of ssetale} that $N'(a)$ is \ssetale{} over $K$. So $a\in N'(a)\subset\pis(L|K)$ as claimed.
\end{proof}

We conclude this section with a result on $\pis(L|K)$ for $\s$-field extensions $L|K$ that will be needed in the proof of our enhanced version of Babbitt's decomposition theorem (Theorem \ref{theo: Babbitt}). This result also illustrates the importance of the $\s$-separability assumption. Example \ref{ex: core not galois} shows that $\pis(L|K)$ is better behaved than $\core(L|K)$.

\begin{prop} \label{prop: pis Galois}
	Let $L|K$ be an extension of $\s$-fields such that $L$ is Galois over $K$ and let $M$ be an intermediate $\s$-field of $L|K$ that is Galois over $K$. Then $\pis(L|M)$ is Galois over $K$.
\end{prop}
\begin{proof}
	It suffices to show that any finite $\s$-field extension $N\subset\pis(L|M)$ of $M$ is contained in a finite Galois $\s$-field extension $N'\subset\pis(L|M)$ of $M$.
	
	Let $a\in N$ be such that $N=M(a)$. The sequence $([K(\s^i(a)):K])_{i\in\mathbb{N}}$ is non-increasing and therefore stabilizes. Let $n\in\mathbb{N}$ be such that $[K(\s^n(a)):K]=[K(\s^{i}(a)):K]$ for $i\geq n$.
	Let $a_1,\ldots,a_m\in L$ denote the conjugates of $\s^n(a)$ over $K$, where $a_1=\s^n(a)$.
	
	Then $N'=M(a_1,\ldots,a_m)$ is a Galois extension of $K$.
	Because $M(a)$ is $\s$-separable over $M$, $\s(a)$ has the same degree as $a$ over $M$. So $M(a)=M(\s(a))$. Inductively, it follows that $M(a)=M(\s^n(a))\subset N'$.
	
	We will next show that $N'$ is a $\s$-field. It suffices to show that $\s(a_1),\ldots,\s(a_m)\in N'$. Let $f$ denote the minimal polynomial of $a_1=\s^n(a)$ over $K$ and let ${\hs f}$ be the polynomial obtained from $f$ by applying $\s$ to the coefficients. Then ${\hs f}(\s^{n+1}(a))=0$ and because $[K(\s^{n+1}(a)):K]=[K(\s^{n}(a)):K]$, the polynomial ${\hs f}$ is irreducible. Since $N'|K$ is Galois and the polynomial ${\hs f}$ has the root $\s(a_1)=\s^{n+1}(a)\in N'$, all the roots of ${\hs f}$ lie in $N'$. But the roots of ${\hs f}$ are $\s(a_1),\ldots,\s(a_m)$. So $N'$ is a $\s$-subfield of $L$.

	It remains to show that $N'$ is $\s$-separable over $M$. 
	As $M(a_1)=M(a)$ is $\s$-separable and finite over $M$ the canonical map ${\hs (M(a_1))}=M(a_1)\otimes_M M\to M(a_1)$ is surjective. So we can write
	\begin{equation} \label{eq: a1}
	a_1=\sum_i \s(f_i(a_1))\lambda_i,
	\end{equation}
	where the $f_i$'s are polynomials over $M$ and $\lambda_i\in M$. Now fix $j\in\{1,\ldots,m\}$ and let $\tau$ be an element of the Galois group of $L|K$ that maps $a_1$ to $a_j$. Recall (Equation (\ref{eq: operation of gal})) that $\tau(\s(b))=\s(\tau^\s(b))$ for $b\in L$. Applying $\tau$ to Equation (\ref{eq: a1}) yields
	\begin{equation} \label{eq: aj}
	a_j=\tau(a_1)=\sum_i\s\left({^{\tau^\s}\! f_i}(\tau^\s(a_1))\right)\tau(\lambda_i)
	\end{equation}
	where ${^{\tau^\s}\! f_i}$ is the polynomial obtained from $f_i$ by applying $\tau^\s$ to the coefficients.
	Since $M$ is Galois over $K$, we see that ${^{\tau^\s}\! f_i}(\tau^\s(a_1))\in N'$ and $\tau(\lambda_i)\in M$. Thus Equation (\ref{eq: aj}) shows that $a_j$ lies in the image of $\psi\colon {\hs(N')}= N'\otimes_M M\to N'$. Therefore $\psi$ is surjective and hence also injective. It follows from Proposition \ref{prop: sseparable} that $N'$ is $\s$-separable over $M$.
\end{proof}

\begin{cor} \label{cor: pis Galois}
	Let $L|K$ be an extension of $\s$-fields such that $L$ is Galois over $K$. Then $\pis(L|K)$ is Galois over $K$.
\end{cor}
\begin{proof}
	This is Proposition \ref{prop: pis Galois} with $M=K$.		
\end{proof}

The following example shows that Corollary \ref{cor: pis Galois} (and thus also Lemma \ref{prop: pis Galois}) does not remain valid with $\core(L|K)$ in place of $\pis(L|K)$. In other words, if $L|K$ is an extension of $\s$-fields such that $L$ is Galois over $K$, then $\core(L|K)$ need not be Galois over $K$.

\begin{ex} \label{ex: core not galois} 
We use the theory of the limit degree and benign extensions explained in Section~\ref{section:babb}.
	Let $x,y,z$ denote $\s$-variables over $\C$, where we consider $\C$ as a $\s$-field by virtue of the identity map, and denote $k=\C\l x,y,z\r$. We fix an extension of $\s\colon k\to k$ to the algebraic closure $\bar{k}$ of $k$. 
	The polynomial $f=T^3+xT^2+yT+z$ has Galois group $S_3$ over $k$. 
Let $a,b,c\in \bar{k}$ denote the roots of $f$. Then $L=k\l a,b,c\r$  
is benign over $k$ with limit degree $6$. 

Let $K=k\l\s(a)\r$. 
Then $K(a)$ is a $\s$-field extension of $K$ and therefore $a\in\core(L|K)$. The element $b\in L$ is conjugate with $a$ over $K$ but $K\l b\r$ is not a finite extension of $K$, since
\begin{multline*}
\ld(K\l b\r|K)\geq \ld(K(a)\l b\r|K(a))=\ld(k\l a,b\r| k\l a\r)\\
= \ld(k\l a,b,c\r | k\l a\r)=\ld(k\l a,b,c\r|k)/ \ld(k\l a\r|k)=6/3=2.
\end{multline*}
	According to Exercise 4.3.27 in \cite{Levin:difference}, one actually has $\ld(K\l b\r|K)=2$. In any case $b\notin\core(L|K)$. This shows that $\core(L|K)$ is not Galois. In particular, $\pis(L|K)\subsetneqq\core(L|K)$.
	
	
%
\end{ex}

\section{Babbitt's decomposition}\label{section:babb}

Babbitt's decomposition is an important structure theorem for finitely $\s$-generated extensions of $\s$-fields $L|K$ such that $L|K$ is Galois (as a field extension). See \cite[Theorem 2.3] {Babbitt:FinitelyGeneratedPathologicalExtensionsOfDifferenceFields} or \cite[Theorem 5.4.13]{Levin:difference}. To state Babbitt's decomposition theorem, we first need to recall the definition of the limit degree, an important numerical invariant of $\s$-field extensions. (See \cite[Section 4.3]{Levin:difference} or \cite[Chapter 5, Section 16, p. 135]{Cohn:difference}.)

%
%

Let $L|K$ be a finitely $\s$-generated extension of $\s$-fields, say $L=K\langle F\rangle$, with $F\subset L$ finite. For $i\in\N$ let
\begin{equation} \label{eq: di}
d_i=\left[K\left(F,\s(F),\ldots,\s^i(F)\right):K\left(F,\s(F),\ldots,\s^{i-1}(F)\right)\right].
\end{equation} Then the sequence $(d_i)_{i\in\N}$ is non-increasing and therefore stabilizes. The eventual value $\ld(L|K)=\lim\limits_{i\to\infty}d_i$ does not depend on the choice of $F$ and is called the \emph{limit degree} of $L|K$.

An extension $L|K$ of $\s$-fields is called \emph{benign} (\cite[Def. 5.4.7]{Levin:difference}) if there exists an intermediate field $K\subset M\subset L$ such that $M$ is finite and Galois over $K$ with $L=K\langle M\rangle$ and $\ld(L|K)=[M:K]$. Assuming that $M$ is finite and Galois with $L=K\langle M\rangle$, the condition $\ld(L|K)=[M:K]$ is equivalent to the condition that $[K(\s^i(M)):K]=[M:K]$ for all $i\in\N$ and the fields $(K(\s^i(M)))_{i\in N}$ are linearly disjoint over $K$.
If $L|K$ is benign, then $L|K$ is Galois and finitely $\s$-generated.

\begin{theo}[Babbitt's Decomposition Theorem, {\cite[Theorem 2.3]{Babbitt:FinitelyGeneratedPathologicalExtensionsOfDifferenceFields}, \cite[Theorem 5.4.13]{Levin:difference}}]
	Let $K$ be an inversive $\s$-field and let $L|K$ be a finitely $\s$-generated $\s$-field extension of $K$ such that $L|K$ is Galois (as a field extension). Then there exists a chain of intermediate $\s$-fields
	$$K\subset L_0\subset L_1\subset\cdots\subset L_n\subset L$$
	such that $L_0=\core(L|K)$, $L_i$ is benign over $L_{i-1}$ for $i=1,\ldots,n$ and $L_n$ and $L$ have identical inversive closures.
\end{theo}

 
  Here we will show that the assumption that $K$ is inversive is not necessary. In \cite{ive-tgs} and \cite{ive-tgsacfa}, the author has shown how to use Babbitt's decomposition for $K$ not necessarily inversive, by imposing the assumption that $L|K$ be $\s$-separable, but in the present paper we proceed in a more elegant way.

Our overall strategy of proof is similar to Babbitt's. However, the original proof uses the assumption that $K$ is inversive in various places and therefore several changes and complements are necessary. We use $\pis(L|K)$ rather than $\core(L|K)$, as well as the notion of a substandard generator rather than the notion of a standard generator (see Definitions \ref{defi: standard generator} and \ref{defi: substandard generator} below). Consequently, the overall structure of the proof now appears somewhat clearer, because we do not need to make the extra step of passing to the inversive closure in order to be able to apply the induction hypothesis.

We also recall the notion of a standard generator (\cite[Def. 5.4.5]{Levin:difference} or \cite[p. 217]{Cohn:difference}).

\begin{defi} \label{defi: standard generator}
Let $L|K$ be a finitely $\s$-generated extension of $\s$-fields such that $L|K$ is Galois. An element $a\in L$ is called a \emph{standard generator of $L|K$} if the following properties are satisfied:
\begin{itemize}
\item $K\l a\r=L$.
\item $[K(a,\s(a)):K(a)]=\ld(L|K)$.
\item The field extension $K(a)|K$ is Galois.
\end{itemize}
A \emph{minimal standard generator} $a$ of $L|K$ is a standard generator $a$ of $L|K$ such that  $[K(a):K]\leq [K(b):K]$ for every standard generator $b$ of $L|K$.
\end{defi}

\begin{defi} \label{defi: substandard generator}
	Let $L|K$ be a finitely $\s$-generated extension of $\s$-fields such that $L|K$ is Galois. An element $a\in L$ is called a \emph{substandard generator of $L|K$} if the following properties are satisfied:
	\begin{itemize}
		\item $L$ is $\s$-radicial over $K\l a\r$.
		\item $[K(a,\s(a)):K(a)]=\ld(L|K)$.
		\item The field extension $K(a)|K$ is Galois.
	\end{itemize}
	A \emph{minimal substandard generator} $a$ of $L|K$ is a substandard generator $a$ of $L|K$ such that  $[K(a):K]\leq [K(b):K]$ for every substandard generator $b$ of $L|K$.
\end{defi}

As explained on page 334 in \cite{Levin:difference}, for any finitely $\s$-generated extension of $\s$-fields $L|K$ such that $L|K$ is Galois, there exists a standard generator of $L|K$. In particular, there exists a minimal substandard generator.

%

Let $L|K$ be a finitely $\s$-generated extension of $\s$-fields such that $L|K$ is Galois and let $a$ be a minimal standard generator of $L|K$. As the sequence of the $d_i$'s (Equation (\ref{eq: di})) is non-increasing, we have
$$[K(a):K]\geq [K(a,\s(a)):K(a)]=\ld(L|K).$$
\begin{lemma} \label{lemma: characterize benign}
Let $L|K$ be a finitely $\s$-generated extension of $\s$-fields such that $L|K$ is Galois and let $a$ be a minimal standard generator of $L|K$. Then $L|K$ is benign if and only if $[K(a):K]=\ld(L|K)$.
\end{lemma}
\begin{proof}
Assume that $L|K$ is benign and let $M\subset L$ be a finite Galois extension of $K$ such that $L=K\langle M\rangle$ and $\ld(L|K)=[M:K]$. Choose $b\in M$ such that  $M=K(b)$. Then $b$ is a minimal standard generator of $L|K$ and $[K(b):K]=\ld(L|K)$.

Conversely, if $a$ is a minimal standard generator of $L|K$ with $[K(a):K]=\ld(L|K)$, then $M=K(a)\subset L$ is a Galois extension of $K$ with $K\langle M\rangle =L$ and $\ld(L|K)=[M:K]$.
\end{proof}

\begin{cor} \label{cor: substandard implies almost benign}
	Let $L|K$ be a finitely $\s$-generated extension of $\s$-fields such that $L|K$ is Galois and let $a$ be a minimal substandard generator of $L|K$. If $[K(a):K]=\ld(L|K)$, then $K\l a\r|K$ is benign.
\end{cor}
\begin{proof}
	A finitely $\s$-generated extension of $\s$-fields has limit degree one if and only if it is finitely generated as a field extension (\cite[Theorem XV, Chapter 5, p. 143]{Cohn:difference}). A finitely $\s$-generated $\s$-radicial extension of $\s$-fields is finitely generated as a field extension and therefore has limit degree one. So $\ld(L|K\langle a\rangle)=1$. Using the multiplicativity of the limit degree (\cite[Theorem 4.3.4]{Levin:difference}) we find
	$$\ld(L|K)=\ld(L|K\langle a\rangle)\cdot\ld(K\langle a\rangle|K)=\ld(K\langle a\rangle|K).$$
	So a standard generator of $K\l a\r$ over $K$ is a substandard generator of $L$ over $K$. Therefore $a$ is a minimal standard generator of $K\l a\r$ over $K$ and it follows from Lemma \ref{lemma: characterize benign} that $K\l a\r|K$ is benign.
\end{proof}

\begin{lemma} \label{lemma: s(a) substandard generator}
	Let $L|K$ be a finitely $\s$-generated extension of $\s$-fields such that $L|K$ is Galois and let $a$ be a substandard generator of $L|K$. Then also $\s(a)$ is a substandard generator of $L|K$. Moreover, if $a$ is a minimal substandard generator of $L|K$ then $[K(\s(a)):K]=[K(a):K]$.
\end{lemma}
\begin{proof}
	Because $L|K\l a\r$ and $K\l a\r|K\l\s(a)\r$ are $\s$-radical, also $L|K\l \s(a)\r$ is $\s$-radicial (Lemma \ref{lemma: sradicial transitive}).
	
	If $f\in K(a)[x]$ is the minimal polynomial of $\s(a)$ over $K(a)$, then ${\hs f}(\s^2(a))=0$ and ${\hs f}\in K(\s(a))[x]$. Thus the degree of $\s^2(a)$ over $K(\s(a))$ is less than or equal to the degree of $\s(a)$ over $K(a)$. Therefore 
	$$[K\left(\s(a),\s^2(a)\right):K(\s(a))]\leq [K(a,\s(a)):K(a)]=\ld(L|K).$$
	On the other hand, $[K\left(\s(a),\s^2(a)\right):K(\s(a))]\geq \ld(K\l\s(a)\r|K)$ and (as in the proof of Corollary \ref{cor: substandard implies almost benign}) $$\ld(L|K)=\ld(L|K\l\s(a)\r)\cdot\ld(K\l\s(a)\r|K)=\ld(K\l\s(a)\r|K).$$ Therefore $[K\left(\s(a),\s^2(a)\right):K(\s(a))]=\ld(L|K)$. Finally $K(\s(a))|K$ is Galois because $K(a)|K$ is Galois. So $\s(a)$ is a substandard generator of $L|K$.
	
	If we assume that $a$ is a minimal substandard generator of $L|K$, the fact that also $\s(a)$ is a substandard generator of $L|K$, implies that $[K(a):K]\leq [K(\s(a)):K]$ and so $[K(a):K]=[K(\s(a)):K]$.	
\end{proof}

\begin{lemma} \label{lemma: si(b)notin K}
		Let $L|K$ be a finitely $\s$-generated extension of $\s$-fields such that $L|K$ is Galois and let $a$ be a minimal substandard generator of $L|K$. Let $b\in K(a)$ with $b\notin K$. Then $\s^i(b)\notin K$ for $i\in\N$. 
\end{lemma}
\begin{proof}
    By Lemma \ref{lemma: s(a) substandard generator} we have $[K(\s^i(a)):K]=[K(a):K]$ for any $i\in\N$.
    So $$1,\s^i(a),\s^i(a)^2,\ldots,\s^i(a)^{n-1}$$ are $K$-linearly independent, where $n$ is the degree of $a$ over $K$.
	Write $b=b_0+b_1a+\cdots+b_{n-1}a^{n-1}$ with $b_0,\ldots,b_{n-1}\in K$. Then $$\s^i(b)=\s^i(b_0)+\s^i(b_1)\s^i(a)+\cdots+\s^i(b_n)\s^i(a)^{n-1}.$$ Suppose $\s^i(b)\in K$. Because $1,\s^i(a),\s^i(a)^2,\ldots,\s^i(a)^{n-1}$ are $K$-linearly independent, it follows that $\s^i(b)=\s^i(b_0)$. So $b=b_0\in K$, in contradiction to the assumption $b\notin K$.
%
\end{proof}

\begin{lemma} \label{lemma: Babbit sradicial}
	Let $L|K$ be a $\s$-radicial extension of $\s$-fields and $a_1,\ldots,a_n$ elements in a $\s$-field extension of $L$ such that $L\l a_1,\ldots,a_{i}\r$ is benign over
	$L\l a_1,\ldots,a_{i-1}\r$ with minimal standard generator $a_{i}$ for $i=1,\ldots,n$. Then there exist integers $r_1,\ldots,r_n\geq 0$ such that
	$K\l \s^{r_1}(a_1),\ldots,\s^{r_i}(a_{i})\r$ is benign over
	$K\l \s^{r_1}(a_1),\ldots,\s^{r_{i-1}}(a_{i-1})\r$ with minimal standard generator $\s^{r_i}(a_{i})$ for $i=1,\ldots,n$.
\end{lemma}
\begin{proof}
We first treat the case $n=1$. So $a:=a_1$ is a minimal standard generator of $L\l a\r$ over $L$. We claim that $K\l\s^j(a)\r|K$ is benign with minimal standard generator $\s^j(a)$ for $j\gg 0$.

Let $f$ denote the minimal polynomial of $a$ over $L$.
Because $L(a)$ is the splitting field of $f$, every root of $f$ can be written as a polynomial in $a$ with coefficients in $L$. Since $L|K$ is $\s$-radicial, every root of $^{\s^{j}}f$ can be written as a polynomial in $\s^j(a)$ with coefficients in $K$ for $j\gg 0$. Therefore $K(\s^j(a))$ is the splitting field of $^{\s^j}f$ for $j\gg0$. Because $f$ is separable also $^{\s^j}f$ is separable. Therefore $K(\s^j(a))|K$ is Galois. 

Because $L$ is $\s$-radicial over $K$ there exists an integer $i$ such that $^{\s^j}f$ has coefficients in $K$ for $j\geq i$. It suffices to show that $^{\s^{j+l}}f$ is irreducible over $K(\s^j(a),\ldots,\s^{j+l-1}(a))$ for $l\in\N$ and $j\geq i$. Suppose the contrary. Then $^{\s^{j+l}}f$ is reducible over $L(a,\ldots,\s^{j+l-1}(a))$. This contradicts the fact that $L\l a\r|L$ is benign with minimal standard generator $a$.
This finishes the case $n=1$.

The general case now follows from the $n=1$ case by induction, with $K\l \s^{r_1}(a_1),\ldots,\s^{r_{n-1}}(a_{n-1})\r$ in place of $K$ and $L\l a_1,\ldots,a_{n-1}\r$ in place of $L$.
\end{proof}

Now we are prepared to prove the enhanced version of Babbitt's decomposition theorem.

\begin{theo} \label{theo: Babbitt}
Let $K$ be a $\s$-field and $L$ a finitely $\s$-generated $\s$-field extension of $K$ such that $L|K$ is Galois. Then there exists a chain of intermediate $\s$-fields
\[K\subset L_0\subset L_1\subset\cdots\subset L_{n-1}\subset L_n\subset L\]
such that $L_0=\pis(L|K)$, $L_{i}$ is benign over $L_{i-1}$ for $i=1,\ldots,n$ and $L$ is $\s$-radicial over $L_n$.
\end{theo}
\begin{proof}
The proof is by induction on $\ld(L|K)$. If $\ld(L|K)=1$, then $L|K$ is finite (\cite[Theorem XVII, p. 144]{Cohn:difference}). Thus, in this case, the claim follows from Lemma \ref{lemma: finite implies sradicial over pis}. So we can assume that $\ld(L|K)>1$.
Replacing $K$ with $\pis(L|K)$, we can assume that $\pis(L|K)=K$ (Lemma \ref{lemma: pispis equals pis}).

First, we also assume that $L|K$ contains no intermediate $\s$-field $M$ such that $M|K$ is Galois and $1<\ld(M|K)<\ld(L|K)$. We will show that the theorem holds with $n=1$.

Let $a$ be a minimal substandard generator of $L|K$. Then $K(a)$ and $K(\s(a))$ are Galois extensions of $K$.
Therefore (see e.g., \cite[Thereom VI, p. 6]{Cohn:difference})
\[[K(a,\s(a)):K]=\frac{[K(a):K]\cdot [K(\s(a)):K]}{[K(a)\cap K(\s(a)):K]}.\]
On the other hand,
\[[K(a,\s(a)):K]=[K(a,\s(a)):K(a)]\cdot[K(a):K]=\ld(L|K)\cdot[K(a):K].\]
Because $[K(\s(a)):K]=[K(a):K]$ by Lemma \ref{lemma: s(a) substandard generator}, the last two equations yield
\[[K(a):K]=\ld(L|K)\cdot [K(a)\cap K(\s(a)):K].\]
Since $\ld(L|K)>1$, we see that $[K(a)\cap K(\s(a)):K]<[K(a):K]$.

Let $b$ be a primitive element of $K(a)\cap K(\s(a))$ over $K$. Then $[K(b):K]<[K(a):K]$ and $K(b)|K$ is Galois. Therefore also $K\l b\r|K$ is Galois. By assumption, $\ld(K\l b\r|K)=1$ or $\ld(K\l b\r|K)=\ld(L|K)$.

Let us first assume that $\ld(K\l b\r|K)=\ld(L|K)$. Then $$[K(b,\s(b)):K(b)]\geq \ld(K\l b\r|K)=\ld(L|K)$$ and therefore
\[[K(b,\s(b)):K]= [K(b,\s(b)):K(b)]\cdot[K(b):K]\geq \ld(L|K)\cdot [K(a)\cap K(\s(a)):K].\]
But $K(b,\s(b))\subset K(\s(a))$ and so
\[[K(b,\s(b)):K]\leq [K(\s(a)):K]= [K(a):K]=\ld(L|K)\cdot [K(a)\cap K(\s(a)):K].\]
Thus $K(b,\s(b))=K(\s(a))$ and $[K(b,\s(b)):K(b)]=\ld(L|K)$. So $K\l b\r=K\l\s(a)\r$, $L$ is $\s$-radicial over $K\l b\r$ and $b$ is a substandard generator of $L|K$. Since $[K(b):K]<[K(a):K]$, this contradicts the fact that $a$ is a minimal substandard generator of $L|K$.

Hence $\ld(K\l b\r|K)=1$. Because $\pis(L|K)=K$, it follows from Corollary \ref{cor: pis sradicial over core} that $\s^n(b)\in K$ for some $n\in\N$. But then $b\in K$ by Lemma \ref{lemma: si(b)notin K}. So $K(a)\cap K(\s(a))=K$ and $[K(a):K]=\ld(L|K)$. Thus $K\l a\r|K$ is benign by Corollary \ref{cor: substandard implies almost benign}. Because $L|K\l a\r$ is $\s$-radicial, we see that the theorem is satisfied with $n=1$ and $L_1=K\l a\r$.

\medskip

Now we assume that there exists an intermediate $\s$-field $K\subset M\subset L$ such that $M|K$ is Galois and $1<\ld(M|K)<\ld(L|K)$. We know from Proposition \ref{prop: pis Galois} that $\pis(L|M)$ is Galois over $K$. Moreover $$\ld(\pis(L|M)|K)=\ld(\pis(L|M)|M)\cdot\ld(M|K)=\ld(M|K),$$
where the last equality uses Remark \ref{rem: strong core vs core} (ii).
Replacing $M$ with $\pis(L|M)$, we may assume that $\pis(L|M)=M$ (Lemma \ref{lemma: pispis equals pis}).
By Theorem \cite[Theorem 4.4.1, p. 292]{Levin:difference} an intermediate $\s$-field of a finitely $\s$-generated $\s$-field extension is finitely $\s$-generated. So $M$ is finitely $\s$-generated over $K$.
Applying the induction hypothesis to $M|K$ yields a sequence of $\s$-fields
$$K\subset L_1\subset\cdots\subset L_n\subset M,$$
where the extensions $L_1|K$ and $L_{i}|L_{i-1}$ ($2\leq i\leq n$) are benign and $M|L_n$ is $\s$-radicial.
Applying the induction hypothesis to $L|M$ yields a sequence of $\s$-fields
$$M\subset M_1\subset\cdots\subset M_m\subset L,$$
where the extensions $M_1|M$ and $M_{i}|M_{i-1}$ ($2\leq i\leq m$) are benign and $L|M_m$ is $\s$-radicial.

Let $a_1$ be a minimal standard generator of $M_1|M$ and for $i=2,\ldots, m$ let $a_i\in M_{i}$ be a minimal standard generator of $M_i|M_{i-1}$. As $M|L_n$ is $\s$-radicial, by Lemma \ref{lemma: Babbit sradicial} there exist $r_1,\ldots,r_{m}\in\N$ such that $L_n\l \s^{r_1}(a_1),\ldots,\s^{r_i}(a_i)\r$ is benign over $L_n\l \s^{r_1}(a_1),\ldots,\s^{r_{i-1}}(a_{i-1})\r$ for $i=1,\ldots,m$.

We have constructed a sequence of benign $\s$-field extensions
$$K\subset L_1\subset\cdots\subset L_n\subset L_n\l \s^{r_1}(a_1)\r\subset\cdots\subset L_n\l \s^{r_1}(a_1),\ldots,\s^{r_m}(a_m)\r$$
inside $L$. It remains to see that $L$ is $\s$-radicial over $L_n\l \s^{r_1}(a_1),\ldots,\s^{r_m}(a_m)\r$. But $L$ is $\s$-radicial over $M_m$ and $M_m=M\l a_1,\ldots,a_m\r$ is $\s$-radicial over $L_n\l \s^{r_1}(a_1),\ldots,\s^{r_m}(a_m)\r$. Therefore $L$ is $\s$-radicial over $L_n\l \s^{r_1}(a_1),\ldots,\s^{r_m}(a_m)\r$.
\end{proof}

\begin{rem}
Let us summarize, using the notation of Theorem~\ref{theo: Babbitt}, the features that make our version of Babbitt's decomposition more versatile than the original.
\begin{enumerate}
\item As already mentioned, it applies to difference field extensions $L|K$ where the base field $K$ is not necessarily inversive.
\item We make the relationship of $L_n$ and $L$ explicit. Indeed, in our version, the tower $K\subset L_n\subset L$ gives a decomposition of $L|K$ into a $\sigma$-separable part and a $\sigma$-radicial part, whereas the original theorem only asserts that the inversive closures of $L_n$ and $L$ coincide.
\end{enumerate}
\end{rem}

\section{Applications}\label{section:appli}

\subsection{A contribution to the study of difference algebraic groups}

As mentioned in the introduction, an initial motivation for studying \ssetale{} \ks-algebras is their use for understanding the difference connected components of difference algebraic groups. In this article we do not want to go through the details of the definition of a difference algebraic group. (The interested reader is referred to \cite{Wibmer:AffineDifferenceAlgebraicGroups} and \cite{Wibmer:Habil}.) We simply rely on the fact that the category of difference algebraic groups is anti-equivalent to the category of \ks-Hopf algebras that are finitely $\s$-generated as \ks-algebras (\cite[Prop. 2.3]{Wibmer:AffineDifferenceAlgebraicGroups}). Here a \ks-algebra $R$ with the structure of a Hopf algebra over $k$ is called a \emph{\ks-Hopf algebra} if the structure maps of the Hopf algebra, i.e., the comultiplication $\Delta\colon R\to R\otimes_k R$, the antipode $S\colon R\to R$ and the counit $\varepsilon\colon R\to k$, are morphisms of difference rings.

\begin{fact}[{\cite[Theorem 4.5]{Wibmer:AffineDifferenceAlgebraicGroups}}]\label{hopf-si-fg}
A \ks-Hopf subalgebra of a \ks-Hopf algebra which is finitely $\s$-generated over $k$, is finitely $\s$-generated over $k$.
\end{fact}
\begin{theo}\label{ssetale-hopf}
 Let $R$ be a \ks-Hopf algebra. Then $\pis(R)$ is a \ks-Hopf subalgebra of $R$. 
 Moreover, if $R$ is finitely $\s$-generated over $k$, then $\pis(R)$ is strongly $\s$-\'etale over $k$.
\end{theo}
\begin{proof}
The first claim follows from Corollary~\ref{monoidal}. For example, to show that $\Delta(\pis(R))\subseteq \pis(R)\otimes_k\pis(R)$, one uses that $\Delta(\pis(R))\subset\pis(R\otimes_k R)$ and that $\pis(R\otimes_k R)=\pis(R)\otimes_k \pis(R)$. The second claim follows from 
Fact~\ref{hopf-si-fg}.
\end{proof}
\begin{rem}
Let $R$ be a \ks-algebra and let $S$ be the union of all \ks-subalgebras of $R$ that are \'{e}tale as $k$-algebras.
The following example shows that
\begin{enumerate}
\item $S$ need not be finitely $\s$-generated over $k$, even if $R$ is;
\item if $R$ is a \ks-Hopf algebra, $S$ need not be a \ks-Hopf subalgebra of $R$.
\end{enumerate}
Thus, working with strongly $\s$-\'etale algebras (and the condition of $\sigma$-separability) in the definition of $\pis$ is crucial for Conjecture~\ref{fin-si-gen}, for obtaining nice functorial properties from Corollary~\ref{monoidal},  and for Theorem~\ref{ssetale-hopf}.
\end{rem}


\begin{ex} \label{ex: core not Hopfalgebra}
	A geometrically more intuitive description of this example can be found in \cite[Ex. 4.2.15]{Wibmer:Habil}. We refer the reader to Sections 2.2 and 2.3 in \cite{Levin:difference} for definitions and basic constructions regarding difference polynomial ideals.
	
	Let $R_1=k\{y\}/[y^2-1,\s(y)-1]$ denote the quotient of the univariate $\s$-polynomial ring modulo the $\s$-ideal $\s$-generated by $y^2-1$ and $\s(y)-1$. 
	Similarly, let $R_2=k\{y\}/[y^2-1]$. If we denote by $z$ the image of $y$ in $R_1$ respectively $R_2$, one may check that the formulas $\Delta(z)=z\otimes z$, $S(z)=z$ and $\varepsilon(z)=1$ define the structure of a \ks-Hopf algebra on $R_1$ respectively $R_2$.
	So $R=R_1\otimes_k R_2$, equipped with the product Hopf algebra structure, is naturally a \ks-Hopf algebra. 
	
Let $S$ be the union of all \ks-subalgebras of $R$ that are \'{e}tale as $k$-algebras. 
	Let us assume that the characteristic of $k$ is not equal to $2$, so that $R$ is a union of \'{e}tale $k$-algebras.
	We have $R_1=ke_1\oplus ke_2$ for orthogonal idempotent elements $e_1,e_2\in R_1$ with $\s(e_1)=1$ and $\s(e_2)=0$. So $$R=R_1\otimes_k R_2=(e_1\otimes R_2)\oplus (e_2\otimes R_2).$$ For any element $a\in e_2\otimes R_2$ we have $\s(a)=0$ and so $k\{a\}=k[a]$ is an \'{e}tale $k$-algebra. This shows that $e_2\otimes R_2$ is contained in $S$. 
In particular, $S$ has infinite dimension as a $k$-vector space. Thus, $S$ is not finitely $\s$-generated over $k$, and Fact~\ref{hopf-si-fg}  shows that it is not a \ks-Hopf subalgebra of $R$. On the other hand, $\pis(R)=k$.

%
\end{ex}

\subsection{An application to compatibility}

Recall (\cite[Def. 5.1.1]{Levin:difference}) that two extensions of $\s$-fields $L|K$ and $L'|K$ are called \emph{compatible} if there exists a $\s$-field extension of $K$ that contains isomorphic copies of $L|K$ and $L'|K$.

The classical compatibility theorem (\cite[Theorem 5.4.22]{Levin:difference}) states that two extensions of $\s$-fields are compatible if and only if their cores are compatible. We will improve on this by showing that the core can be replaced by $\pis$. Note that for a $\s$-field extension $L|K$ one has $\pis(L|K)\subset\core(L|K)$.

\begin{lemma} \label{lemma: sradicial compatible}
	Let $L|K$ and $L'|K$ be two extensions of $\s$-fields. If $L|K$ is $\s$-radicial then $L|K$ and $L'|K$ are compatible.
\end{lemma}
\begin{proof}
	Let $L'^*$ denote the inversive closure of $L'$. (See \cite[Def. 2.1.6]{Levin:difference}.) Then $L'^*$ contains an inversive closure of $K$. On the other hand, because $L|K$ is $\s$-radicial, the inversive closure $L^*$ of $L$ is an inversive closure of $K$. So, by the uniqueness of the inversive closure, there exists a $K$-embedding of $L^*$ into $L'^*$, which restricts to a $K$-embedding of $L$ into $L'^*$. Thus $L'^*$ contains isomorphic copies of $L|K$ and $L'|K$.
\end{proof}

\begin{theo}
	Let $L|K$ and $L'|K$ be two extensions of $\s$-fields. Then $L|K$ and $L'|K$ are compatible if and only if $\pis(L|K)|K$ and $\pis(L'|K)|K$ are compatible.
\end{theo}
\begin{proof}
	A possible line of proof would be to follow the proof of the classical compatibility theorem, but to use Theorem \ref{theo: Babbitt} instead of the classical version of Babbitt's decomposition. However, using Corollary \ref{cor: pis sradicial over core} and Lemma \ref{lemma: sradicial compatible}, we can easily deduce a proof from the classical compatibility theorem:
	
	A $\s$-field extension containing isomorphic copies of $L|K$ and $L'|K$ obviously also contains isomorphic copies of $\pis(L|K)|K$ and $\pis(L'|K)|K$.
	To prove the non-trivial implication, it suffices, by the classical compatibility theorem, to show that $\core(L|K)$ and $\core(L'|K)$ are compatible. By assumption, there exists a $\s$-field extension $M$ of $K$, containing $\pis(L|K)|K$ and $\pis(L'|K)|K$. By Corollary \ref{cor: pis sradicial over core} and Lemma \ref{lemma: sradicial compatible} there exists a $\s$-field extension $M_1$ of $\pis(L|K)$ containing $M$ and $\core(L|K)$. Similarly, by Corollary~\ref{cor: pis sradicial over core} and Lemma \ref{lemma: sradicial compatible} there exists a $\s$-field extension $M_2$ of $\pis(L'|K)$ containing $M_1$ and $\core(L'|K)$. Thus $M_2$ contains $\core(L|K)$ and $\core(L'|K)$.
\end{proof}

\bibliographystyle{alpha}
\bibliography{bibdata}

\def\cprime{$'$}
\begin{thebibliography}{DVHW14}

\bibitem[Bab62]{Babbitt:FinitelyGeneratedPathologicalExtensionsOfDifferenceFields}
Albert~E. Babbitt, Jr.
\newblock Finitely generated pathological extensions of difference fields.
\newblock {\em Trans. Amer. Math. Soc.}, 102:63--81, 1962.

\bibitem[Bou72]{Bourbaki:commutativealgebra}
Nicolas Bourbaki.
\newblock {\em Elements of mathematics. {C}ommutative algebra}.
\newblock Hermann, Paris, 1972.
\newblock Translated from the French.

\bibitem[Bou90]{Bourbaki:Algebra2}
N.~Bourbaki.
\newblock {\em Algebra. {II}. {C}hapters 4--7}.
\newblock Elements of Mathematics (Berlin). Springer-Verlag, Berlin, 1990.
\newblock Translated from the French by P. M. Cohn and J. Howie.

\bibitem[CHP02]{acfa2}
Zo{\'e} Chatzidakis, Ehud Hrushovski, and Ya'acov Peterzil.
\newblock Model theory of difference fields. {II}. {P}eriodic ideals and the
  trichotomy in all characteristics.
\newblock {\em Proc. London Math. Soc. (3)}, 85(2):257--311, 2002.

\bibitem[Coh65]{Cohn:difference}
Richard~M. Cohn.
\newblock {\em Difference algebra}.
\newblock Interscience Publishers John Wiley \& Sons, New York-London-Sydeny,
  1965.

\bibitem[DVHW14]{DiVizioHardouinWibmer:DifferenceGaloisTheory}
Lucia Di~Vizio, Charlotte Hardouin, and Michael Wibmer.
\newblock Difference {G}alois theory of linear differential equations.
\newblock {\em Adv. Math.}, 260:1--58, 2014.

\bibitem[Hru04]{Hrushovski:elementarytheoryoffrobenius}
Ehud Hrushovski.
\newblock The elementary theory of the {F}robenius automorphisms, 2004.
\newblock arXiv:math/0406514v1, updated version available from
  http://www.ma.huji.ac.il/$\sim$ ehud/.

\bibitem[Lev08]{Levin:difference}
Alexander Levin.
\newblock {\em Difference algebra}, volume~8 of {\em Algebra and Applications}.
\newblock Springer, New York, 2008.

\bibitem[Mil12]{Milne:BasicTheoryOfAffineGroupSchemes}
James~S. Milne.
\newblock Basic theory of affine group schemes, 2012.
\newblock Available at www.jmilne.org/math/.

\bibitem[OW15]{OvchinnikovWibmer:SigmaGalois}
Alexey Ovchinnikov and Michael Wibmer.
\newblock {$\sigma$}-{G}alois theory of linear difference equations.
\newblock {\em Int. Math. Res. Not. IMRN}, (12):3962--4018, 2015.

\bibitem[Roq18]{julien-mahler}
Julien Roques.
\newblock On the algebraic relations between {M}ahler functions.
\newblock {\em Trans. Amer. Math. Soc.}, 370(1):321--355, 2018.

\bibitem[Spr09]{Springer:LinearAlgebraicGroups}
T.~A. Springer.
\newblock {\em Linear algebraic groups}.
\newblock Modern Birkh\"auser Classics. Birkh\"auser Boston, Inc., Boston, MA,
  second edition, 2009.

\bibitem[{Sta}14]{stacks-project}
The {Stacks Project Authors}.
\newblock {S}tacks {P}roject.
\newblock \url{http://stacks.math.columbia.edu}, 2014.

\bibitem[Tom15]{ive-tgsacfa}
Ivan Toma{\v{s}}i{\'c}.
\newblock Galois stratification and {ACFA}.
\newblock {\em Ann. Pure Appl. Logic}, 166(5):639--663, 2015.

\bibitem[Tom16]{ive-tgs}
Ivan Toma{\v s}i{\'c}.
\newblock Twisted {G}alois stratification.
\newblock {\em Nagoya Math. J.}, 222(1):1--60, 2016.

\bibitem[Tom18]{ive-direct-tgs}
Ivan Toma{\v s}i{\'c}.
\newblock Direct twisted {G}alois stratification.
\newblock {\em Ann. Pure Appl. Logic}, 169(1):21--53, 2018.

\bibitem[Wat79]{Waterhouse:IntroductiontoAffineGroupSchemes}
William~C. Waterhouse.
\newblock {\em Introduction to affine group schemes}, volume~66 of {\em
  Graduate Texts in Mathematics}.
\newblock Springer-Verlag, New York, 1979.

\bibitem[Wib]{Wibmer:AffineDifferenceAlgebraicGroups}
Michael Wibmer.
\newblock Affine difference algebraic groups.
\newblock arXiv:1405.6603.

\bibitem[Wib10]{Wibmer:thesis}
Michael Wibmer.
\newblock {\em Geometric difference {G}alois theory}.
\newblock PhD thesis, Heidelberg, 2010.
\newblock http://www.ub.uni-heidelberg.de/archiv/10685.

\bibitem[Wib15]{Wibmer:Habil}
Michael Wibmer.
\newblock Affine difference algebraic groups, 2015.
\newblock Habilitation thesis, available at
  \url{http://www.math.upenn.edu/~wibmer/habilWibmer.pdf}.

\end{thebibliography}
\end{document}